\documentclass[12pt,a4paper,english,reqno,a4paper]{amsart}
\usepackage[T1]{fontenc}
\usepackage[utf8]{inputenc}
\usepackage{color}
\usepackage{verbatim}
\usepackage{amstext}
\usepackage{amsthm}
\usepackage{stmaryrd}
\usepackage{setspace}
\makeatletter

\numberwithin{equation}{section}
\numberwithin{figure}{section}
\theoremstyle{plain}
\newtheorem{thm}{\protect\theoremname}[section]
  \theoremstyle{remark}
  \newtheorem{rem}[thm]{\protect\remarkname}
  \theoremstyle{plain}
  \newtheorem{lem}[thm]{\protect\lemmaname}
  \theoremstyle{definition}

\usepackage{amsfonts}\usepackage{amsthm}
\usepackage{epsfig}\usepackage{mathrsfs}
\usepackage{amstext}
\usepackage{esint}
\usepackage[T1]{fontenc}
\usepackage[utf8]{inputenc}

\@ifundefined{definecolor}
 {\usepackage{color}}{}

\makeatletter
\newcommand{\Rmnum}[1]{\expandafter\@slowromancap\romannumeral#1@}\makeatother

\numberwithin{equation}{section}

\allowdisplaybreaks

\newcommand{\set}[1]{\left\{#1\right\}}

\newcommand{\defs}{:=}

\DeclareMathOperator*{\dist}{dist}

\newcommand{\dif}{\mathrm{d}}

\DeclareSymbolFont{lettersA}{U}{pxmia}{m}{it}
\DeclareMathSymbol{\piup}{\mathord}{lettersA}{"19}

\newcommand{\mcc}{\mathcal{C}}

\makeatother

\makeatother

\usepackage{babel}
  \providecommand{\definitionname}{Definition}
  \providecommand{\lemmaname}{Lemma}
  \providecommand{\remarkname}{Remark}
\providecommand{\theoremname}{Theorem}

\begin{document}

\title[Uniqueness of transonic shock Solutions in an Expanding nozzle]{ Uniqueness of Transonic Shock Solutions for Two-Dimensional Steady Compressible Euler Flows in an Expanding Nozzle}

\author{Beixiang Fang}

\author{Xin Gao}

\author{Wei Xiang}

\address{B.X. Fang: School of Mathematical Sciences, MOE-LSC, and SHL-MAC, Shanghai
Jiao Tong University, Shanghai 200240, China }

\email{\texttt{bxfang@sjtu.edu.cn}}

%

\address{X. Gao: Department of Applied Mathematics, The Hong Kong Polytechnic University, Hong Kong, China}
\email{\texttt{xingao@polyu.edu.hk}}
%

\address{W. Xiang: Department of Mathematics, City University of Hong Kong, Kowloon, Hong Kong, China}

\email{\texttt{weixiang@cityu.edu.hk }}

\keywords{2-D; steady Euler system; transonic shocks; receiver pressure; uniqueness;}
\subjclass[2010]{35A01, 35A02, 35B20, 35B35, 35B65, 35J56, 35L65, 35L67, 35M30, 35M32, 35Q31, 35R35, 76L05, 76N10}

\date{\today}


 \email{}
 \email{}
 \email{}
\begin{abstract}
In this paper, we are trying to show the uniqueness of transonic shock solutions in an expanding nozzle under certain conditions and assumptions on the boundary data and the shock solution. The idea is to compare two transonic shock solutions and show that they should coincide if the perturbation of the nozzle is
sufficiently small. To this end, a condition on the pressure of the flow across the shock front is proposed, such that a priori estimates for the subsonic flow behind the shock front could be established without the assumption that it is a small perturbation of the unperturbed uniform subsonic state. With the help of these estimates, the uniqueness of the position of the intersection point between the shock front and the nozzle boundary could be further established by demonstrating the monotonicity of the solvability condition for the elliptic sub-problem of the subsonic flow behind the shock front. Then, via contraction arguments, two transonic shock solutions could be verified to coincide as the perturbation is small, which leads to the uniqueness of the transonic shock solution.
\end{abstract}

\maketitle
\section{Introduction and main results}
In this paper, we are concerned with the uniqueness of transonic shocks for steady Euler flows in a finite expanding two-dimensional (2-d) nozzle.
If the nozzle is an expanding angular sector or a diverging cone, under the assumption that the flow parameters only depend on the radius, Courant-Friedrichs established in \cite{CR} the unique existence of the transonic shock solutions if the value of the receiver pressure lies in a certain interval. These transonic shock solutions have been shown to be structurally stable for generic small perturbations of boundary data and the geometry of the nozzle, for instance, by Chen in \cite{SC2009} and by Li-Xin-Yin in \cite{LXY2009,LiXY2009,LiXY2010,LiXinYin2011PJM,LXY2013}. Recently, Fang-Xin in \cite{FB63} showed that, with the same prescribed receiver pressure at the exit, there may exist more than one shock solutions for steady Euler flows in a general nozzle with both expanding and contracting portions, which yields the non-uniqueness of the transonic shock solutions for general cases.
However, if the nozzle is strictly expanding, only one shock solution has been established via the methods developed in \cite{FB63}.
Thus, one may wonder whether it is possible to establish the uniqueness of the transonic shock solutions for steady Euler flows in a strictly expanding nozzle with the given pressure at the exit. In this paper, we are going to show that this uniqueness can be true under certain conditions and assumptions on the boundary data.

\subsection{The transonic shock problem in a nozzle with a receiver pressure condition at the exit}
Let the nozzle be bounded in
$$\mathcal{D} \defs \{(x_1,x_2 )\in {\mathbb{R}}^2: 0 < x_1 < L,\  0 < x_2  < \phi(x_1)\}$$
 with the entrance $E_0$, the exit $E_L$ as well as the walls $W_2$ and $W_4$ (see Figure \ref{fig:1}), \emph{i.e.}
\begin{align}
  E_0 &\defs \{(x_1,x_2)\in {\mathbb{R}}^2 : x_1 = 0,\ 0<x_2< 1\},\label{E0EQ}\\
 W_2 &\defs \{(x_1,x_2)\in {\mathbb{R}}^2 : 0< x_1 < L,\ x_2 =0 \},\label{W2EQ} \\
  E_L&\defs \{(x_1 ,x_2)\in {\mathbb{R}}^2 : x_1 =L, \ 0<x_2 <\phi(L)\},\label{ELEQ}\\
 W_4&\defs \{(x_1 ,x_2)\in {\mathbb{R}}^2 : 0< x_1 <L,\ x_2 =\phi(x_1)\},\label{W4EQ}
\end{align}
where $\phi(x_1)\in \mathcal{C}^{2,\alpha}[0,L]$ is a given function. In addition,  for convenience of calculations, we assume that the lower bound of the nozzle is flat (\emph{i.e.}\eqref{W2EQ}). If the lower bound of the nozzle is a perturbation of a flat one, it can be treated similarly.
\begin{figure}[!h]
\centering
\includegraphics[width=0.5\textwidth]{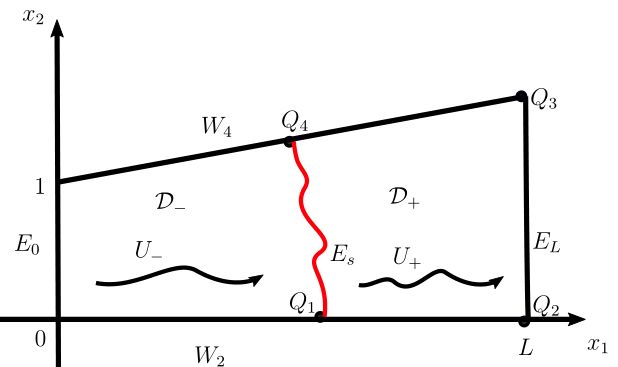}
\caption{The steady Euler flow with a transonic shock front in an expanding nozzle.\label{fig:1}}
\vspace{-0.2cm}
\end{figure}

Assume that the flow in the nozzle $\mathcal{D}$ is governed by the following 2-d steady Euler equations:
\begin{align}
&\partial_{x_1} (\rho u) + \partial_{x_2} (\rho v)=0,\label{eq1}\\
&\partial_{x_1} (\rho u v ) + \partial_{x_2} (\rho v^2 + p) = 0,\\
&\partial_{x_1} (\rho u^2 + p) + \partial_{x_2} (\rho u v)=0,\\
&\partial_{x_1} (\rho u \mathrm{B}) + \partial_{x_2} (\rho v \mathrm{B}) = 0\label{eq4},
\end{align}
where $u$ and $v$ represent the horizontal velocity and vertical velocity respectively, $\rho$ is the density, $p$ is the pressure, $\mathrm{B} = \frac12(u^2 + v^2) + i$ is the Bernoulli constant, with $i =e+\frac{p}{\rho}$ the enthalpy and $e$ the internal energy.
Moreover, the fluid is supposed to be a polytropic gas with the state equation
$p=A(s){\rho}^{\gamma}$,
where $s$ is the entropy, $\gamma>1$ is the adiabatic exponent, and $A(s)=R e^{\frac{s}{c_v}}$ with $c_v$ being the specific heat at constant volume and $R$ is a positive constant.
Let $U = (p,\theta, q, s)^\top$ be the state of the fluid, where $\theta = \arctan\displaystyle\frac{v}{u}$ is the flow angle and $q= \sqrt{u^2 + v^2}$ is the speed. Moreover, let $c = \sqrt{\displaystyle\frac{\gamma p}{\rho}}$ be the sonic speed, and let $M = \displaystyle\frac{q}{c}$ be the Mach number. When $M < 1$
(respectively, $M > 1$), the flow is called a subsonic flow (respectively a supersonic
flow).
In this paper, we are concerned with the flow pattern involving a single transonic shock in the nozzle. Let $E_s \defs\set{x_1 = \varphi(x_2)}$ be the position of the shock front (see Figure \ref{fig:1}), then the following Rankine-Hugoniot conditions (which we abbreviate as R-H conditions hereafter) should hold on $E_s$:
\begin{align}
&[\rho u]-\varphi'[\rho v]=0,\label{5}\\
&[\rho u v]- \varphi'[p+\rho v^2]=0,\label{6}\\
&[p+\rho u^2]- \varphi'[\rho u v]=0,\label{7}\\
&[\mathrm{B}]=0.\label{8}
\end{align}

\begin{rem}\label{shockremark}

For a fixed point $Q = (\varphi(x_2), x_2)$ on the shock front $E_s$,
let $U_-(Q) = (p_-(Q), u_-(Q), v_-(Q), \rho_-(Q))$ be the supersonic state of the flow ahead of $E_s$ at the point $Q$, then in \cite[p.306-p.309]{CR},
it has been shown that the R-H conditions \eqref{5}-\eqref{8} yield that all possible states $U(Q)$ behind the shock front $E_s$ at $Q$ form a curve in the phase space, which is called shock polar. (See Figure \ref{fig:2} for projection of the shock polar on $\theta-p$ plane.)
In particular, one has the following important relation on $\theta$ and $p$ (cf.\cite[p.347]{CR}):
\begin{align}\label{op}
  \tan(\theta - \theta_-) = \pm \frac{\frac{p}{p_-} -1}{\gamma M_-^2 - (\frac{p}{p_-} -1)}\sqrt{{\frac{\frac{2\gamma}{\gamma+1}(M_-^2 - 1) - (\frac{p}{p_-} -1)}{\frac{p}{p_-} + \frac{\gamma -1}{\gamma +1}}}},
\end{align}
where $c_- =  \sqrt{\frac{\gamma p_-}{\rho_-}}$, $M_- =\frac{q_-}{c_-}$ is the Mach number of the flow ahead of $E_s$ at $Q$.

\begin{figure}[!h]
\centering
\includegraphics[width=0.4\textwidth]{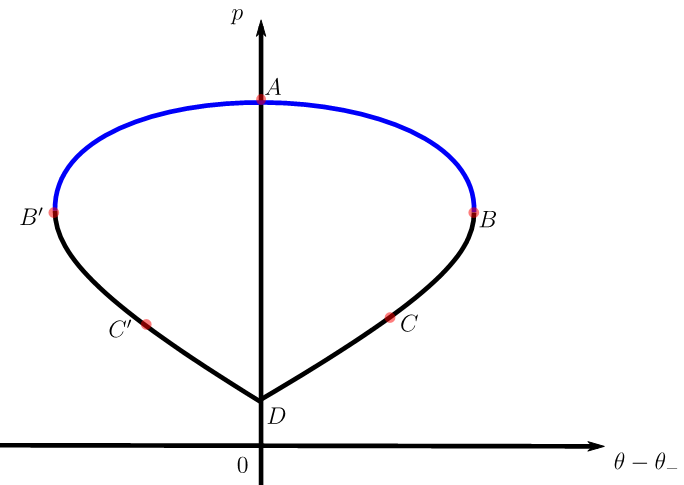}
\caption{The shock polar on $\theta$-$p$ plane.\label{fig:2}}
\end{figure}
There are several important points on the $\theta-p$ shock polar:
$A$ is the point where $p$ reaches its maximum, $B= (\theta_*+\theta_-, p_*)$ is the critical point where $ p_*$ is the critical pressure and  $\theta$ reaches its maximum.
The point $C = (\theta_{sonic}+\theta_-, p_{sonic})$ has the special property that the state
of the gas behind the shock-front is exactly sonic, that is, $M = 1$. For all points on the shock polar with $p > p_{sonic}$ (that is, on the open arc $ \widehat{CAC'}$), the states behind
the shock front are all subsonic. The points $B'$
and $C'$ are the mirror images of $B$ and $C$ with respect to the line $\theta = \theta_-$.
Moreover, on the shock front $E_s$, one has
\begin{align}\label{spu}
u =  u_- - \frac{p-p_-}{\rho_- u_-},\quad
\rho =\frac{(\gamma+1)p+(\gamma-1)p_-}{(\gamma-1)p+(\gamma+1)p_-} \rho_-.
\end{align}
(It should be pointed out that the notations used in this Remark are independent and have no relations to the ones in other parts of the paper.)
\end{rem}
The domain $\mathcal{D} $ is divided by the shock front $E_s$ into two parts: the supersonic region $\mathcal{D}_-$ and the subsonic region $\mathcal{D}_+$ respectively (see Figure \ref{fig:1}), which are:
	\begin{align}
	& \mathcal{D}_- = \{ (x_1,x_2)\in \mathbb{R}^2 : 0 < x_1 < \varphi(x_2), \ 0 < x_2 < \phi(x_1)\},\label{9}\\
	& \mathcal{D}_+ = \{ (x_1,x_2)\in \mathbb{R}^2 : \varphi(x_2) < x_1 < L,\ 0 < x_2 <\phi(x_1)\}.\label{10}
	\end{align}
  Thus, the problem determining the states of the flow with a single transonic shock in the nozzle can be mathematically formulated as a free boundary problem described below.

\textbf{Transonic shock problem} {\bf {$\llbracket\textit{TSP}\rrbracket$}}.

 Try to find a transonic shock solution $(U_-, U_+, \psi)$ to Euler equations \eqref{eq1}-\eqref{eq4} in domain $\mathcal{D}$ such that:
\vspace{-0.15cm}
\begin{enumerate}
\item In $ \mathcal{D}_- $, $ U =  U_-(x_1,x_2)$ is the supersonic solution and satisfies Euler equations \eqref{eq1}-\eqref{eq4} with boundary conditions:
\vspace{-0.1cm}
	\begin{align}
	U_-=& U_-(0,x_2),&\quad &\text{on} &\quad  E_0,\label{11}\\
	\theta_- =& 0, &\quad &\text{on} &\quad W_2 \cap\overline{\mathcal{D}_-},\label{12}\\
\theta_- =& \arctan\phi'(x_1), &\quad &\text{on} &\quad W_4 \cap\overline{\mathcal{D}_-},\label{13}
	\end{align}
where $U_-(0,x_2)\in \mcc^{2,\alpha}([0,1])^4$ is a given vector function for the state of the flow at the entrance of the nozzle;

\item In $ \mathcal{D}_+ $, $ U =  U_+(x_1,x_2)$ is the subsonic solution and satisfies Euler equations \eqref{eq1}-\eqref{eq4} with boundary conditions
\vspace{-0.1cm}
	\begin{align}
	\theta_+ =& 0, &\quad &\text{on} &\quad W_2 \cap\overline{\mathcal{D}_+},\label{14}\\
\theta_+ = & \arctan\phi'(x_1), &\quad &\text{on} &\quad W_4 \cap\overline{\mathcal{D}_+},\label{15}\\
	p_+ =& \mathcal{P}(L, x_2),&\quad &\text{on}&\quad E_L,\label{16}
	\end{align}
where $ \mathcal{P}\in \mathcal{C}^{2,\alpha}(\overline{\mathbb{R}}_+)$ is a given function for the receiver pressure at the exit of the nozzle;
\item On the shock front $E_s$, the R-H conditions $\eqref{5}$-$\eqref{8}$ hold.
\end{enumerate}
\begin{rem}
\eqref{12}-\eqref{15} are slip boundary conditions on the nozzle walls.
\end{rem}

When $\phi(x_1) = \phi_0(x_1) \equiv 1$, it is well known that there exist planar transonic shock solutions $(\overline{U}_-, \overline{U}_+, \overline{\psi}(x_2) )$ to the problem {\bf {$\llbracket\textit{TSP}\rrbracket$}},
where $\overline{U}_- = ( \bar{q}_-,0, \bar{p}_-, \bar{s}_-)^\top$, $\overline{U}_+ = ( \bar{q}_+,0,\bar{p}_+,  \bar{s}_+)^\top$ and the position of the plane shock front $\overline{\psi}(x_2)\equiv \bar{x}_s$ with $\bar{x}_s$ being an arbitrary value in $(0,L)$, satisfying the R-H conditions:
\begin{align}
&[\bar{\rho} \bar{q}] = \bar{\rho}_{+} \bar{q}_{+} - \bar{\rho}_{-} \bar{q}_{-}=0,\label{bg1}\\
&[\bar{p}+\bar{\rho} \bar{q}^2] =  (\bar{p}_{+}+\bar{\rho}_{+} \bar{q}_{+}^2) - (\bar{p}_{-}+\bar{\rho}_{-} \bar{q}_{-}^2)=0,\label{bg2}\\
&[\bar{\mathrm{B}}] =  \bar{\mathrm{B}}_{+} - \bar{\mathrm{B}}_{-}=0.\label{bg3}
\end{align}

As $\phi$ is a small perturbation of $\phi_0$, by applying the methods developed in \cite{FB63} by Fang-Xin, it can be established the existence of the solution to the problem {\bf {$\llbracket\textit{TSP}\rrbracket$}} as the pressure at the exit satisfies certain conditions. Furthermore, if $\phi'(x_1)>0$, namely, the nozzle is strictly expanding, only one solution to the problem {\bf {$\llbracket\textit{TSP}\rrbracket$}} can be established. Motivated by the uniqueness of transonic shocks for radial flows in expanding nozzles established in \cite{CR,LXY2009,LiXY2009,LiXY2010,LiXinYin2011PJM,LXY2013}, one may anticipate that the shock solution obtained via the methods in \cite{FB63} is the only one with the prescribed boundary data. In this paper, without the assumption that the flow behind the shock front is a small perturbation of $\overline{U}_+$, we are trying to figure out sufficient conditions and assumptions on the boundary data and shock solutions such that the uniqueness of transonic shocks can be established for steady flows in a strictly expanding nozzle.
\subsection{Existence of transonic shock solutions}
Let $\phi(x_1) \defs 1+ x_1 \tan\sigma$, where $\sigma>0$ is a sufficiently small constant such that $\phi(x_1)$ is a small perturbation of $\phi_0$. Obviously
\begin{align}\label{specialcondition}
  \arctan\phi'(x_1) = \sigma.
\end{align}
Let the state of the flow at the entrance
\begin{align}
   U_-(0,x_2)\defs\overline{U}_- + \sigma \cdot {U}_0(x_2),
\end{align}
where $ U_0(x_2) =\big( p_0(x_2), \theta_0(x_2), q_0(x_2), s_0 (x_2)\big)^\top \in \mcc^{2,\alpha}([0,1])^4$. In addition, the incoming data $U_-(0,x_2)$ satisfies the following compatibility conditions at the points $(0,0)$ and $(0,1)$ respectively:
\begin{align}
   &\theta_0(0) = 0,
   \quad \theta_0 (1) =\sigma,\label{CC0}\\
&p_0'(0) = 0, \quad \theta_0'(1) + \frac{M_-^2(0,1) -1}{\rho_-(0,1) q_-^2(0,1) \tan\sigma}p_0'(1)=0,\label{CC1}\\
&(M_-^2(0,0) -1)\theta_0^{''} (0) + \Big( \frac{1}{\gamma c_v}s_0'(0)  -\frac{2}{q_-(0,0)}q_0'(0)\Big)\theta_0'(0) =0,\\
 &\Big( M_-^2(0,1) -1 + \tan^2\sigma\Big)\theta_0^{''}(1) + \frac{2(M_-^2(0,1)-1)\tan\sigma}{\rho_-(0,1) q_-^2(0,1)}p_0^{''}(1)\notag\\
 &+\frac{ M_-^2(0,1) -1 + \tan^2\sigma}{ M_-^2(0,1) -1}\Big( \frac{1}{\gamma c_v}s_0'(1) - \frac{2}{q_-(0,1)} q_0'(1)   \Big)  \theta_0^{'}(1)\notag\\
  & +\frac{ 2 \widehat{\mathcal{M}_-} (0,1)\sin\sigma}{(M_-^2(0,1)-1)^2\cos^3\sigma} \big( \theta_0^{'}(1) \big)^2 =0,\label{CC2}
\end{align}
where
\begin{align*}
  \widehat{\mathcal{M}_-} (0,1)\defs& M_-^2(0,1)-1 - \gamma M_-^4(0,1)\sin^2\sigma \notag\\
  &+ M_-^2(0,1)\big(  1+M_-^4(0,1) -2M_-^2(0,1)\big)\cos^2\sigma.
  \end{align*}
Let the receiver pressure at the exit
\begin{align}\label{specialconditionP}
 \mathcal{P}(L,x_2) \defs  \bar{p}_+ +  \sigma P_e,
\end{align}
where $P_e$ is a constant and satisfying
\begin{align}\label{Peinterval}
 \frac{1-\bar{M}_+^2}{\bar{\rho}_+^2 \bar{q}_+^3}\eta_0 P_e\in \big( (1-{\kappa})L+ \mathcal{G}, \,\, L + \mathcal{G} \big),
\end{align}
where
\begin{align}
\mathcal{G} \defs& \int_{0}^{\eta_0}(1 - \kappa)p_0 - \frac{1}{\bar{\rho}_-\bar{q}_-}\big(\kappa_1(p_0 + \bar{\rho}_-\bar{q}_- q_0) + \kappa_2 s_0   \big)\dif \eta,\\
 \eta_0 \defs& \int_0^1  (\rho_- q_- \cos\theta_-) (0,x_2) \dif x_2,\\
  {\kappa} \defs& \Big(\frac{\gamma - 1}{\gamma \bar{p}_+} + \frac{1}{\bar{\rho}_+ {\bar{q}_+}^2}\Big) [\bar{p}] >0,\label{kappaeq1}\\
  \kappa_1\defs &-\frac{1}{\bar{\rho}_-\bar{q}_-^2}\Big( \kappa + \frac{[\bar{p}]}{\bar{\rho}_+\bar{q}_+^2}\Big)    - \frac{\gamma-1}{\gamma \bar{p}_+}\Big(1-\frac{\bar{\rho}_+}{\bar{\rho}_-} \Big),\label{kappa1eq1}\\
  \kappa_2\defs&  \frac{1}{\gamma c_v}\Big( \kappa + \frac{\bar{c}_-^2 - \bar{c}_+^2}{\bar{c}_+^2}\Big).\label{kappa2eq1}
\end{align}

 By applying the methods developed in \cite{FB63} by Fang-Xin, the existence of transonic shock solutions with $\sigma>0$ being sufficiently small can be established
 by showing the existence of shock solutions reformulated under the Lagrange transformation,
which is defined by
  $$(\xi, \eta)\defs \big(x_1, \, \int_{(0,0)}^{(x_1,x_2)}\rho u(s,t) \dif t - \rho v(s,t) \dif s   \big).$$
Then the Euler equations \eqref{eq1}-\eqref{eq4} can be rewritten as the following form: (cf. \cite[(2.49)-(2.52)]{FB63})
\begin{align}
&\partial_{{\eta}} {p} - \displaystyle\frac{\sin {\theta}}{{\rho} {q}}\partial_{{\xi}} {p} + {q} \cos {\theta} \partial_{{\xi}} {\theta} =0,\label{LC1}\\
&\partial_{{\eta}} {\theta} - \displaystyle\frac{\sin {\theta}}{{\rho} {q}} \partial_{{\xi}} {\theta} - \displaystyle\frac{\cos {\theta}}{{\rho} {q}}\displaystyle\frac{1 - {{M}}^2}{{\rho} {{q}}^2 } \partial_{{\xi}}{ p} =0,\label{LC2}\\
&\rho q \partial_\xi q +  \partial_\xi p = 0,\label{LC3}\\
&\partial_\xi s=0,\label{LC4}
\end{align}
or, equivalently, the following conservation form:
\begin{align}
  &\partial_\xi \Big(\displaystyle\frac{1}{\rho q\cos\theta}\Big) - \partial_\eta
(\tan\theta) = 0,\label{Conf1}\\
&\partial_\xi (q\sin\theta) + \partial_\eta p =
0,\label{Conf2}\\
&\partial_\xi (\frac12 q^2 + i)=0,\label{Conf3}\\
&\partial_\xi s=0.\label{Conf4}
\end{align}
Moreover, the upper boundary $W_4$ becomes $ \set{\eta = \eta_0} $ with
\begin{align*}
  \eta_0 = \int_0^1  (\rho_- q_- \cos\theta_-) (0,x_2) \dif x_2.
\end{align*}
Let ${\Gamma}_s \defs\{ ({{\xi}}, {{\eta}})\in \mathbb{R}^2 : {\xi} = {\psi}({\eta}),\,\,0 <{\eta} <\eta_0 \}$ be the position of a shock front, then the R-H conditions \eqref{5}-\eqref{8}
 across the shock front are reformulated as
\begin{align}
  G_1(U_+,U_-) \defs & \Big[\frac{1}{\rho u}\Big][p]-\Big[\frac{v}{u}\Big][v] = 0,\label{g0}\\
G_2(U_+,U_-) \defs & \Big[u+\frac{p}{\rho u}\Big][p]+\Big[\frac{pv}{u}\Big][v]=0,\label{g1}\\
G_3(U_+,U_-) \defs & \Big[\frac12 q^2 + i \Big]=0,\label{g3}\\
G_4(U_+,U_-, \psi') \defs & [v] - \psi'[p]=0.\label{g4}
\end{align}
Under the Lagrange transformation, the domain $\mathcal {D}$ becomes
\begin{equation}\label{tras}
\begin{aligned}
{\Omega} = \{ ({\xi}, {\eta})\in \mathbb{R}^2 : 0 < {\xi} < L, \,\,\, 0 < {\eta} < \eta_0\},
\end{aligned}
\end{equation}
and it is separated by a shock front ${\Gamma}_s$
into two parts (see Figure \ref{fig:3L}):
the supersonic region and subsonic region, denoted by
\begin{align}
{\Omega}_- =& \{ ({\xi}, {\eta})\in \mathbb{R}^2 : 0 < {\xi} < \psi({\eta}), \,\,\, 0 < {\eta} < \eta_0\},\label{LOmega-}\\
{\Omega}_+ =& \{ ({\xi}, {\eta})\in \mathbb{R}^2 : \psi({\eta}) < {\xi} < L, \,\,\, 0 < {\eta} < \eta_0\},\label{LOmega+}
\end{align}
respectively.
The boundaries $E_0$, $W_2$, $E_L$, $W_4$ become
\begin{align*}
&{\Gamma}_1 = \{ ({{\xi}}, {{\eta}})\in \mathbb{R}^2 : {\xi} =0,\,\,\, 0 <{\eta} <\eta_0 \},\\
&{\Gamma}_2 = \{ ({{\xi}}, {{\eta}})\in \mathbb{R}^2 :  0 <{\xi} < L, \,\,\,{\eta} =0 \},\\
&{\Gamma}_3 = \{ ({{\xi}}, {{\eta}})\in \mathbb{R}^2 : {\xi}  = L ,\,\,\, 0<{\eta} <\eta_0 \},\\
&{\Gamma}_4 = \{ ({{\xi}}, {{\eta}})\in \mathbb{R}^2 :  0 < {\xi} < L , \,\,\,{\eta} =\eta_0 \}.
\end{align*}
\begin{figure}[!h]
\centering
\includegraphics[width=0.45\textwidth]{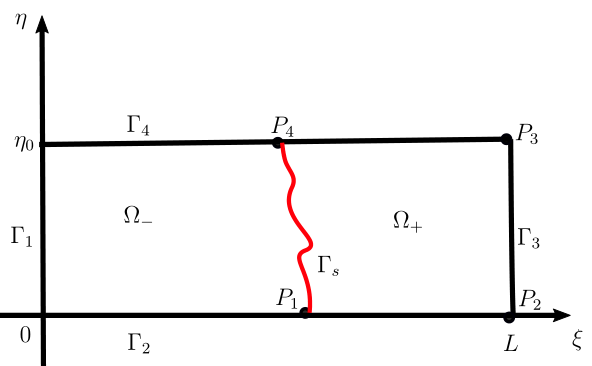}
\caption{The steady Euler flow with a transonic shock front in the Lagrangian coordinates.\label{fig:3L}}
\end{figure}

Then the problem {\bf {$\llbracket\textit{TSP}\rrbracket$}} becomes the following problem:

\textbf{Transonic shock problem} {\bf {$\llbracket\textit{TSPL}\rrbracket$}}.

 Try to find a transonic shock solution $(U_-, U_+, \psi)$ to Euler equations \eqref{LC1}-\eqref{LC4} in domain $\Omega$ such that:
\begin{enumerate}
	\item $ U_{-}(\xi,\eta) $ is the supersonic solution and satisfies the equations  \eqref{LC1}-\eqref{LC4} in $ \Omega_- $ with the boundary conditions:
	\begin{align}
	&U_-= \overline{U}_- + \sigma U_0 (Y_0(\eta)),&\quad  &\text{on} \quad   \Gamma_1,\label{eq808}\\
&\theta_- =0, &\quad &\text{on} \quad  \Gamma_{2}\cap\overline{\Omega_-},\\
	&\theta_- = \sigma,&\quad &\text{on} \quad  \Gamma_{4}\cap\overline{\Omega_-},\label{eq909}
	\end{align}
	where
\begin{align}
  Y_0(\eta) = \int_0^\eta \displaystyle\frac{1}{\rho_- q_- \cos\theta_-(0,s) } \dif s;
\end{align}
	
	\item  $ U_+(\xi,\eta) $ is the subsonic solution and satisfies the equations  \eqref{LC1}-\eqref{LC4} in $ \Omega_+ $ with the following boundary conditions:
	\begin{align}
&\theta_+ = 0,&\quad &\text{on} \quad  \Gamma_{2}\cap\overline{\Omega_+},\\
	&\theta_+ = \sigma, &\quad &\text{on} \quad \Gamma_{4}\cap\overline{\Omega_+},\label{eq911}\\
	&p_+ =\bar{p}_+ +  \sigma P_e,&\quad &\text{on}\quad \Gamma_3;\label{eq910}
	\end{align}

	\item On the shock front $ \Gamma_s$, the Rankine-Hugoniot conditions $\eqref{g0}$-$\eqref{g4}$ hold for the states $(U_{-}, U_+)$.

\end{enumerate}



Based on the methods developed in \cite{FB63} by Fang-Xin,
one can establish the existence of the solutions to the problem {\bf {$\llbracket\textit{TSPL}\rrbracket$}} as the pressure at the exit satisfies the condition \eqref{Peinterval}. As the following theorem shows.
\begin{thm}\label{thm1:existence}
  Suppose that \eqref{specialcondition}-\eqref{Peinterval} hold. Then there exists a sufficiently small positive constant $\sigma_0$, only depends on $\overline{U}_{\pm}$ and $L$, such that for any $0< \sigma \leq \sigma_0$, there exists a solution $(U_-^{\sharp}, U_+^{\sharp} ; \varphi_{\sharp})$ to the problem {\bf {$\llbracket\textit{TSPL}\rrbracket$}}, and the solution satisfies the following estimates:
	\begin{align}
	&\| U_-^{\sharp} - \overline{U}_- \|_{\mcc^{2,\alpha}(\Omega_-)}  \leq C_-^{\sharp} \sigma,\label{existU-}\\
	&\|{U}_+^{\sharp} -\overline{U}_+\|_{(\Omega_+;P_i)}+ \|\varphi_{\sharp}'\|_{\mcc^{0,\alpha} (\Gamma_s)}\leq C_+^{\sharp} \sigma,\label{existU+}
	\end{align}
	where the constant $C_-^{\sharp}$ depends on $\overline{U}_-$, $U_0$, $\gamma$ and $\alpha$, and the constant $C_+^{\sharp}$ depends on  $\overline{U}_{\pm}$, $U_0$, $\gamma$, $L$, $P_e$ and $\alpha$.

\end{thm}

\begin{rem}

The norms of \eqref{existU-} and \eqref{existU+} are defined as follows:
\begin{equation}
   \|u\|_{\mcc^{k,\alpha}(\Omega)} : = \sum_{|\mathbf{m}|\leq k} \sup\limits_{\mathbf{x}\in\Omega}|D^{\mathbf{m}} u(\mathbf{x})|+ \sum_{|\mathbf{m}| =  k} \sup\limits_{\mathbf{x},\mathbf{y}\in\Omega; \mathbf{x}\neq \mathbf{y}}\frac{|D^{\mathbf{m}} u(\mathbf{x}) - D^{\mathbf{m}} u(\mathbf{y})|}{|\mathbf{x}-\mathbf{y}|^\alpha}.
\end{equation}
In addition, for any $\mathbf{x}$, $\mathbf{y}$ $\in$ $\Omega$, define
\begin{equation}
d_{\mathbf{x}}: = \dist(\mathbf{x}, \Gamma),\,\,\, \text{and} \,\,\,
d_{\mathbf{x},\mathbf{y}}:= \min (d_{\mathbf{x}},d_{\mathbf{y}}).
\end{equation}
Let $\alpha\in(0,1)$ and $\delta\in \mathbb{R}$, we define:
\begin{align}
&[u]_{k ,0;\Omega}^{(\delta;\Gamma)} : = \sup\limits_{\mathbf{x}\in\Omega, |\mathbf{m}| = k}\Big(d_{\mathbf{x}}^{\max(k + \delta,0)}|D^{\mathbf{m}} u(\mathbf{x})| \Big),\\
&[u]_{k,\alpha;\Omega}^{(\delta;\Gamma)} : = \sup\limits_{{\mathbf{x},\mathbf{y}}\in\Omega, \mathbf{x}\neq\mathbf{y}, |\mathbf{m}| = k}\left(d_{\mathbf{x},\mathbf{y}}^{\max( k +\alpha+\delta,0)}\displaystyle\frac
{|D^{\mathbf{m}} u(\mathbf{x})- D^{\mathbf{m}} u(\mathbf{y})|}{|\mathbf{x} - \mathbf{y}|^\alpha}\right),\\
& \| u \|_{k,\alpha;\Omega}^{(\delta;\Gamma)} :  =\sum_{i =0}^{k} [u]_{i,0;\Omega}^{(\delta;\Gamma)} + [u]_{k,\alpha;\Omega}^{(\delta;\Gamma)}.
\end{align}
For $ U=(p,\theta,q,s)^\top $, define
\begin{align}\label{eq099}
\|{U} \|_{(\Omega_+;P_i)}: = \|p \|_{1,\alpha;\Omega_+}^{(-\alpha;\{P_i\})}
+\|\theta\|_{1,\alpha;\Omega_+}^{(-\alpha;\{P_i\})}+ \|q \|_{1,\alpha;\Omega_+}^{(-\alpha;\{\overline{\Gamma_2^+}\cup \overline{\Gamma_4^+}\})}
+\|s \|_{1,\alpha;\Omega_+}^{(-\alpha;\{\overline{\Gamma_2^+}\cup \overline{\Gamma_4^+}\})},
\end{align}
where
\begin{align*}
\Gamma_2^+ \defs \Gamma_2 \cap \Omega_+,\quad \Gamma_4^+\defs \Gamma_4\cap \Omega_+,
\end{align*}
$\Gamma_2$, $\Gamma_4$ and the points $P_i$ $(i = 1,2,3,4)$ are defined in Figure \ref{fig:3L}.

\end{rem}

  \begin{rem}
  In the supersonic domain $\Omega_-$, applying the theorems of Chapter 2-Chapter 3 in \cite{LT1985} and compatibility conditions \eqref{CC0}-\eqref{CC2}, it is easy to check that there exists a unique solution $U_-\in \mcc^{2,\alpha}(\overline{\Omega_-})$ satisfying
  \begin{align}\label{zz0}
    \| U_- - \overline{U}_- \|_{\mcc^{2,\alpha}(\overline{\Omega_-})}  \leq C_- \sigma,
  \end{align}
where the constant $C_-$ depends on $\overline{U}_-$, $U_0$ and $\alpha$.
Thus, it suffices to show the uniqueness of the shock positions and the subsonic solutions behind the shock front.
\end{rem}

\subsection{Uniqueness of transonic shock solutions for the problem {\bf {$\llbracket\textit{TSPL}\rrbracket$}}}

In this paper, without the assumption that the flow behind the shock front is a small perturbation of the state $\overline{U}_+$, we are trying to figure out sufficient conditions and assumptions on the boundary data and shock solutions such that the uniqueness of transonic shocks to the problem {\bf {$\llbracket\textit{TSPL}\rrbracket$}} can be established for steady flows in a strictly expanding nozzle. As the following theorem shows:
\begin{thm}\label{thm2}
Suppose that \eqref{specialcondition}-\eqref{Peinterval} hold and
$(U_-^{\aleph}, U_+^{\aleph}; \varphi_{\aleph})$ is a transonic shock solution to the problem
{\bf {$\llbracket\textit{TSPL}\rrbracket$}} with $U_-^{\aleph} =(p_-^{\aleph}, \theta_-^{\aleph}, q_-^{\aleph}, s_-^{\aleph})$ and $U_+^{\aleph} =(p_+^{\aleph}, \theta_+^{\aleph}, q_+^{\aleph}, s_+^{\aleph})$.
Let $\epsilon>0$ be a small constant. Assume that $p_+\geq \bar{p}_* + \epsilon$ on the shock front $\Gamma_s$, the Mach number $M_+^2\leq 1-\epsilon$ and the horizontal velocity $u_+>0$ in the domain $\Omega_+$. Then there exists a sufficiently small constant $\hat{\sigma}>0$, only depending on $\overline{U}_{\pm}$, $L$ and $\epsilon$, such that for any $0< \sigma \leq \hat{\sigma}$, $(U_-^{\aleph}, U_+^{\aleph}; \varphi_{\aleph})$ coincides with $(U_-^{\sharp}, U_+^{\sharp} ; \varphi_{\sharp})$. That is, there exists at most one transonic shock solution to the problem
{\bf {$\llbracket\textit{TSPL}\rrbracket$}}. Moreover, $(U_+^{\aleph}, \varphi_{\aleph}')$ satisfies the estimate
\begin{align}
		\|{U}_+^{\aleph} -\overline{U}_+\|_{(\Omega_+;P_i)}+ \|\varphi_{\aleph}'\|_{\mcc^{0,\alpha} (\Gamma_s)}\leq C_+ \sigma,
\end{align}
where the constant $C_+$ only depends on $\overline{U}_{\pm}$, $U_0$, $\gamma$, $L$, $P_e$ and $\alpha$.
\end{thm}

\begin{rem}
$\bar{p}_*$ is a positive constant corresponding to the second component of the points $B$ and $B_*$ (see Remark \ref{shockremark}) for $\theta_-=0$, which is determined by the uniform supersonic state $\overline{U}_-=( \bar{q}_-,0, \bar{p}_-, \bar{s}_-)^\top$.
\end{rem}

In order to prove Theorem \ref{thm2},
one of the key difficulties is to establish the \emph{a priori} estimates for the subsonic solutions $U_+$ behind the shock front,
without the assumption that the flow behind the shock front is a small perturbation of $\overline{U}_+$.
To this end, one needs to analyze the state of the flow field before and after the shock front carefully.
Motivated by the ideas of Fang-Liu-Yuan in \cite{FL},
a condition on the pressure of the flow across the shock front is proposed, and the {\it a priori} estimates for the states $U_+$ can be established with further careful analysis on the elliptic sub-problem for the subsonic flow behind the shock front.

Another difficulty is to prove the uniqueness of the positions of intersection points between the shock front and the upper wall of the nozzle. It is observed in \cite{FB63} that the position of the intersection points will be determined by the solvability condition for the elliptic sub-problem for the subsonic flow behind the shock front, which is a nonlinear equation for the position of the point.
In addition, there may exist more than one shock solutions for steady Euler flows in a general nozzle with both expanding and contracting portions, which yields the
non-uniqueness of the transonic shock solutions for general case.
Therefore, one needs to demonstrate that there exist at most one solution to this nonlinear equation to show the uniqueness of the position of the intersection point.
Based on the \emph{a priori} estimates of $U_+$, one can further establish that the monotonicity of the solvability condition with respect to the position of the intersection point. Thus, the uniqueness of the position of the intersection point can be obtained. Then, one can compare two transonic shock solutions, and show that they should coincide if the perturbation of the nozzle is sufficiently small, which implies the uniqueness of the transonic shock solutions.

Thanks to great efforts made by many mathematics in the past decades, up to now, for steady potential flows or Euler flows, the existence, stability, and uniqueness of transonic shocks have been studied from
different viewpoints, for instance, see \cite{BF,GCM2007,GK2006,GM2003,CY,CS2005,SC2009,CR,EmbidGoodmanMajda1984,FG,FG3, FX,FB63,LXY2009,LiXY2009,LiXY2010,LiXinYin2011PJM,LXY2013,LXY2016, LYY,Liu1982ARMA,Liu1982CMP,ParkRyu_2019arxiv,Yong,WS94,XYY,XinYin2005CPAM,YH} and references therein. In \cite{CR}, Courant and Friedrichs first gave a systematic analysis from the viewpoint of nonlinear partial differential equations.
Since then, in order to establish a rigorous mathematical analysis for the flow
pattern, various nonlinear PDE models and different boundary conditions have been
proposed. For the unsteady transonic gas flows governed by the quasi-one-dimensional models, for instance, in \cite{Liu1982ARMA,Liu1982CMP},
T.P. Liu proved the existence of shocks solutions for certain given Cauchy data, and established a stability theory for them.
In \cite{EmbidGoodmanMajda1984}, Embid-Goodman-Majda showed that, in general, there exist more than one shock solutions for the steady quasi-one-dimensional model. As to
the steady multi-dimensional models such as potential equations or the Euler system, there are two typical kinds of nozzles. One is an expanding nozzle of an angular
sector or a diverging cone. In \cite{CR}, Courant and Friedrichs established the unique existence of a transonic shock solution in such a nozzle with given constant pressure at the exit.
Based on this shock solution, the well-posedness of shock solutions in an expanding nozzle has been established, with prescribed pressure at the exit as suggested by Courant and Friedrichs, for instance, see \cite{SC2009,LXY2009,LiXY2009,LiXY2010,
LiXinYin2011PJM,LXY2013,Yong,WS94} and reference therein.
The other is a flat nozzle with two parallel walls. The well-posedness of transonic shock solutions have been established for generic small perturbations of boundary data and the geometry of the nozzle, for instance, see \cite{GK2006,GM2003,FB63,
XYY,XinYin2005CPAM} and reference therein.

 There are also studies on the uniqueness of steady flows in a nozzle. Liu-Yuan proved in \cite{LYY} the uniqueness of transonic potential flows in 3-d divergent nozzles.
And in \cite{CY}, for given uniform supersonic upstream flows in a straight duct, Chen-Yuan established the uniqueness of solutions with a transonic shock in a duct for steady potential flows. Moreover, in \cite{FL}, Fang-Liu-Yuan proved that the piece-wise constant transonic shock solution is the unique solution in the flat nozzle and the straight wedge respectively. Motivated by the ideas of \cite{FL},  Fang-Xiang in \cite{FX} proved the transonic shock solution is unique under small perturbations of the wedge boundary.

\subsection{Outline of the paper.}

The paper is organized as follows. In Section 2, the \emph{a priori} estimates for the subsonic solution behind the shock front will be established without the assumption that it is a small perturbation of the uniform subsonic state.
Based on the \emph{a priori }estimates,
in Section 3, the uniqueness of transonic shock solutions can be established via contraction arguments. It is worth mentioning that one of the key steps in the arguments is to prove the uniqueness of intersection points between the shock front and the nozzle walls.



%

\section{A priori estimates for subsonic solutions behind the shock front}

In this section, the \emph{a priori} estimates for subsonic solutions $U_+$ behind the shock front $\Gamma_s$ will be established.
As the following lemma shows. (To simplify the notations, in the subsonic domain behind the shock front, we will drop the subscript ``$ + $'' to let $U=(p, \theta, q, s)^\top$ in the following arguments.)


\begin{lem}\label{elliptic}
 Under the assumptions of Theorem \ref{thm2}, there exists a sufficiently small positive constant $\sigma_1$, depending on $\overline{U}_{\pm}$, $L$ and $\epsilon$, such that for any $0< \sigma \leq \sigma_1$, the solutions $( U; \psi)$ to the problem {\bf{$\llbracket\textit{TSPL}\rrbracket$}} satisfy the following estimate:
\begin{align}\label{es}
		\|{U} -\overline{U}_+\|_{(\Omega_+;P_i)} + \|\psi' \|_{\mcc^{0,\alpha}(\Gamma_s)}\leq C_+ \sigma,
\end{align}
where the constant $C_+$ only depends on $\overline{U}_{\pm}$, $U_0$, $\gamma$, $L$, $P_e$ and $\alpha$.
\end{lem}

\begin{proof}

The proof is divided into six steps.

\emph{Step 1}: First, by applying the Bernoulli's equation \eqref{Conf3} and the R-H condition \eqref{g3}, one can deduce that
\begin{align}\label{boundedness1}
  \|u\|_{L^\infty(\Omega_+\cup \Gamma_s)} +  \|v\|_{L^\infty(\Omega_+\cup \Gamma_s)}+  \|c\|_{L^\infty(\Omega_+\cup \Gamma_s)}  \leq C_{01},
\end{align}
where the constant $C_{01}$ depends on $\overline{U}_-$, $U_0$, $\gamma$ and $L$.
Employing \eqref{boundedness1}, \eqref{zz0} and \eqref{op}-\eqref{spu} on the shock front, it follows that
$p$ and $\rho$ are bounded on the shock front, then by applying the state equation $p=A(s)\rho^{\gamma}$, one can deduce that $s$ is bounded on the shock front, then \eqref{Conf4} yields that
\begin{align}\label{boundedness2}
   \|s\|_{L^\infty(\Omega_+\cup \Gamma_s)} \leq C_{02},
\end{align}
where the constant $C_{02}$ depends on $\overline{U}_-$, $U_0$, $\gamma$ and $L$. In addition, by applying \eqref{boundedness1}, \eqref{boundedness2} and the relation $p = \frac{c^2 }{\gamma}(\frac{c^2}{\gamma A(s)})^{\frac{1}{\gamma-1}}$, one has
\begin{equation}
  \|p\|_{L^\infty(\Omega_+\cup \Gamma_s)} \leq C_{03},
\end{equation}
where the constant $C_{03}$ depends on $\overline{U}_-$, $U_0$, $\gamma$ and $L$.
Therefore,
\begin{align}
  \|U\|_{L^\infty(\Omega_+\cup \Gamma_s)} \leq C_{04},
\end{align}
where the constant $C_{04}$ depends on $\overline{U}_-$, $U_0$, $\gamma$ and $L$.
Thus, by standard elliptic estimates (cf.\cite{GI2011}) for $(p,\theta)$ and the method of characteristic (cf.\cite{LT1985}) for $(q,s)$, one can obtain
 \begin{align}\label{upperest}
  \|U\|_{(\Omega_+;P_i)} \leq C_{0},
\end{align}
where the constant $C_{0}$ depends on $\overline{U}_-$, $U_0$, $\gamma$ and $L$. Moreover, by \eqref{g4} and \eqref{upperest}, one has
\begin{align}\label{uppsi}
   \|\psi' \|_{\mcc^{0,\alpha} (\Gamma_s)}\leq C_{0s},
\end{align}
where the constant $C_{0s}$ depends on $C_{0}$, $\overline{U}_-$, $U_0$, $\gamma$ and $L$.

\emph{Step 2}: In this step, we derive the boundary value problem of elliptic equation with respect to $\theta$ in $\Omega_+$. From equations $\eqref{LC1}$ and \eqref{LC2}, one has
\begin{align}
&\partial_\xi p = \displaystyle\frac{\rho^2 q^3}{(1-M^2)\cos\theta}\partial_\eta \theta - \displaystyle\frac{\rho q^2 \sin\theta}{(1 - M^2)\cos\theta }\partial_\xi \theta,\label{eq180}\\
&\partial_\eta p= \displaystyle\frac{\rho q^2 \sin\theta}{(1- M^2)\cos\theta}\partial_\eta \theta - \displaystyle\frac{q(1-M^2\cos^2\theta)}{(1-M^2)\cos\theta}\partial_\xi \theta,\label{eq140}
\end{align}
which can be rewritten as the following form:
\begin{equation}\label{eq19}
  Dp = \displaystyle\frac{1}{(1-M^2)\cos\theta}K D\theta,
\end{equation}
where the matrix
\[ K = \begin{pmatrix}
   -\rho q^2 \sin\theta& \rho^2 q^3\\
  -q(1-M^2\cos^2\theta)& \rho q^2 \sin\theta \\
\end{pmatrix} .\]
Applying the assumptions $M^2\leq 1-\epsilon$ and $u>0$ in the subsonic domain behind the shock front in Theorem \ref{thm2}, it is easy to see that $\det K = \rho^2 q^4( 1- M^2)\cos^2\theta \neq 0$.
Taking $(\partial_\eta, -\partial_\xi)$ on the both sides of $\eqref{eq19}$, one can obtain a second order elliptic equation of $\theta$ in the divergence form as below:
\begin{equation}\label{eq23}
 \mathcal{N} \theta := \sum_{k,j = 1}^{2} \partial_j(a_{kj}(U) \partial_k \theta)= 0,
\end{equation}
where
\begin{align*}
  a_{11}(U) =& \displaystyle\frac{q(1-M^2 \cos^2\theta)}{(1-M^2)\cos\theta},\quad
  a_{12} (U)= a_{21}(U) = - \displaystyle\frac{\rho q^2\sin\theta}{(1 - M^2)\cos\theta},\\
  a_{22}(U) =& \displaystyle\frac{\rho^2 q^3}{(1-M^2)\cos\theta},
\end{align*}
with the property
\begin{align}\label{eq25}
  &a_{11}(U)\varsigma_1^2 + 2 a_{12}(U)\varsigma_1 \varsigma_2 + a_{22}(U)\varsigma_2^2\notag\\
   =& \displaystyle\frac{1}{(1-M^2)\cos\theta} \notag\\
   &\cdot \Big\{ \Big(\sqrt{q(1-M^2\cos^2\theta)} \varsigma_1 - \displaystyle\frac{\rho q^2 \sin\theta}{\sqrt{q(1-M^2\cos^2\theta)}}\varsigma_2\Big)^2 +  \displaystyle\frac{\rho^2 q^3 (1-M^2)\cos^2\theta}{1-M^2\cos^2\theta}\varsigma_2^2 \Big\}\notag\\
  \geq& \lambda (\varsigma_1^2 + \varsigma_2^2),
\end{align}
where the positive constant $\lambda > 0$ depends on $\epsilon$, and $\varsigma=(\varsigma_1, \varsigma_2)\in \mathbb{R}^2\backslash{(0,0)}$.
Moreover, on the exit $\Gamma_3$, one has
\begin{equation}\label{eq31}
  \partial_{\vec{\tau}_3} p= \sigma \partial_{\vec{\tau}_3} P_e =0,
\end{equation}
where $\vec{\tau}_3 = (0,1)$ is the unit tangential vector of the boundary $\Gamma_3$.
By employing \eqref{eq19}, one has
\begin{equation}\label{eq32}
   0=  \partial_{\vec{\tau}_3} p= \displaystyle\frac{1}{(1-M^2)\cos\theta} {\vec{\tau}_3}\cdot  K D\theta\defs  \vec{\ell}_e \cdot D\theta,
\end{equation}
where
$$\vec{\ell}_e \defs  \displaystyle\frac{1}{(1-M^2)\cos\theta} \big(-q (1-M^2 \cos^2\theta), \rho q^2 \sin\theta\big).$$
Further calculation yields that
\begin{align}\label{reexit}
 \mathcal{A}_e \partial_{\vec{n}_3} \theta + \mathcal{B}_e \partial_{\vec{\tau}_3} \theta = 0,\quad \text{on} \quad \Gamma_3
\end{align}
where ${\vec{n}_3} = (-1,0)$ is the unit inner normal vector of $\Gamma_3$, and
\begin{align}
  &\mathcal{A}_e  =\vec{\ell}_e \cdot \vec{n}_3 = \displaystyle\frac{q( 1-M^2\cos^2\theta)}{(1-M^2)\cos\theta}  > 0,\\
   &\mathcal{B}_e  = \vec{\ell}_e \cdot \vec{\tau}_3 = \displaystyle\frac{\rho q^2 \sin\theta}{(1-M^2)\cos\theta}.
\end{align}
On shock front $\Gamma_s$, \eqref{op} can be rewritten as
\begin{align}\label{112}
  H_1(\theta,p;U_-) = 0.
\end{align}
It is easy to see that $(\partial_{\theta} H_1, \partial_{p} H_1)$ is the normal vector of the level set $H_1=0$ in the $(\theta,p)$-plane (see Figure \ref{fig:2}). If $\theta$ attains its maximum $\theta_*$, then $p_*$ can be expressed by $U_-$, let
\begin{align}
  p_* = H_2(U_-).
\end{align}
It yields that
\begin{align}
  p_* = H_2(\overline{U}_-) + \nabla_{U_-} H_2 \cdot ( U_- - \overline{U}_-) = \bar{p}_* + O(1)\sigma,
\end{align}
where we use \eqref{zz0}, and $O(1)$ is a bounded function.
Thus, for sufficiently small constant $\sigma>0$, one has
\begin{align}
  p>\bar{p}_* + \epsilon > {p}_*, \quad \text{on}\quad \Gamma_s,
\end{align}
which yields that
\begin{align}\label{varepsilon0eq}
  \partial_p H_1 \geq \varepsilon_0 >0,
\end{align}
where the constant $\varepsilon_0$ depends on $\epsilon$ and $\overline{U}_\pm$.
Moreover, by taking the tangential derivative on both sides of equation \eqref{112}, one has
\begin{align}\label{deffs}
   \partial_{\theta} H_1\cdot  \partial_{\vec{\tau}_s} \theta + \partial_{p} H_1 \cdot \partial_{\vec{\tau}_s} p = - \nabla_{U_-} H_1 \cdot \partial_{\vec{\tau}_s} U_- : =  f_s,
\end{align}
where ${\vec{\tau}_s} = \displaystyle\frac{(-\psi',-1)}{\sqrt{1+ (\psi')^2}}$.
Furthermore, applying \eqref{eq19}, then \eqref{deffs} implies that
\begin{align}\label{fs1}
  \partial_{\theta} H_1\cdot  \partial_{\vec{\tau}_s} \theta + \displaystyle\frac{\partial_p H_1}{(1-M^2)\cos\theta} {\vec{\tau}_s}\cdot  K \cdot D\theta =  f_s,
\end{align}
which can be rewritten as
\begin{align}\label{ww}
  \vec{\ell}_s \cdot D\theta= f_s,
\end{align}
where
\begin{align}
  \vec{\ell}_s =&  \partial_{\theta} H_1\cdot {\vec{\tau}_s} +  \displaystyle\frac{\partial_p H_1}{(1-M^2)\cos\theta} {\vec{\tau}_s}\cdot  K  = \displaystyle\frac{1}{\sqrt{1+ (\psi')^2}}( {\ell}_s^{(1)}, \,{\ell}_s^{(2)}),
\end{align}
and
\begin{align}
   {\ell}_s^{(1)} =&-\psi' \partial_{\theta} H_1 + \frac{q(\psi'\rho q \sin\theta + 1- M^2 \cos^2 \theta)\cdot \partial_p H_1}{(1-M^2)\cos\theta},\label{ell1}\\
    {\ell}_s^{(2)} =& -\partial_{\theta} H_1 - \frac{q(\psi'\rho^2 q^2 + \rho q \sin\theta)\cdot \partial_p H_1}{(1-M^2)\cos\theta}\label{ell2}.
\end{align}
In addition, \eqref{fs1} can be rewritten as
\begin{equation}\label{eq29}
     \mathcal{A}_s\partial_{\vec{n}_s} \theta+ \mathcal{B}_s\partial_{\vec{\tau}_s} \theta = f_s,\quad \text{ on} \quad \Gamma_s,
\end{equation}
where ${\vec{n}_s}=\displaystyle\frac{(1, -\psi')}{\sqrt{1+ (\psi')^2}}$, and
     \begin{align}
 \mathcal{A}_s: =&\frac{q\Big( (\psi' \rho q + \sin\theta)^2 + (1-M^2 )\cos^2\theta   \Big) \partial_p H_1}{(1+(\psi')^2)(1-M^2)\cos\theta},\label{eq33}\\
 \mathcal{B}_s: =& \partial_{\theta} H_1 + \frac{q\Big(\rho q  (1 - (\psi')^2)\sin\theta + \psi' (\rho^2 q^2 -1 + M^2 \cos^2\theta)\Big) \partial_p H_1}{(1+(\psi')^2)(1-M^2)\cos\theta}.\label{eq34}
\end{align}
By applying \eqref{varepsilon0eq}, one has
  \begin{align}\label{As}
    \mathcal{A}_s \geq \varepsilon_1>0,
  \end{align}
where the constant $\varepsilon_1 >0 $ depends on $\epsilon$, $\varepsilon_0$ and $\overline
{U}_\pm$.
Moreover, on the nozzle walls $\Gamma_2$ and $\Gamma_4$, one has
\begin{align}
	&\theta= 0, \quad &\text{on} &\quad &\Gamma_2 \cap\overline{\Omega_+},\label{14l}\\
&\theta = \sigma, \quad &\text{on} &\quad &\Gamma_4 \cap\overline{\Omega_+}.\label{15u}
	\end{align}

\emph{Step 3}:
In this step, we will establish the \emph{a priori} estimates for $\theta$.
First, applying the strong maximum principle, the minimum and maximum values of $\theta$ only attain on the boundaries of $\Omega_+$. On the exit $\Gamma_3$, employing the boundary condition \eqref{reexit} and the Hopf's lemma, $\theta$ cannot attain its maximum or minimum values on the boundary $\Gamma_3$. On the boundaries $\Gamma_2$ and $\Gamma_4$, $\theta$ satisfies the conditions \eqref{14l}-\eqref{15u}. In addition, on the free boundary $\Gamma_s$, the oblique derivative boundary condition \eqref{eq29} holds. Therefore, it suffices to analyze the values of $\theta$ in the neighborhood of the boundary $\Gamma_s$.
First, we consider the domain near the point $P_4 = ({\xi}_*,\eta_0)$ (see Figure \ref{fig:3L}), where ${\xi}_*\in (0,L)$ is a constant.
Let
\begin{align}
\xi - {\xi}_* = r \cos\theta,\quad
\eta - \eta_0  = r\sin\theta.
\end{align}
For a constant $r_0\in\mathbb{R}$ with $0<r_0<\frac{1}{2}$, define
\begin{align}
  \Omega_1(r_0): =& \{(\xi,\eta)\in \Omega_+: |(\xi,\eta) - P_4|< r_0\},\\
  \Gamma_N (r_0) \defs&  \Gamma_s\cap \partial  \Omega_1(r_0),\\
  \Gamma_D (r_0)  \defs& \Gamma_4\cap \partial  \Omega_1(r_0).
\end{align}
In light of the method in \cite{GM86, GM13, Miller}, let
\begin{align}
  \mathcal{F}_1 = \|f_s\|_{L^\infty(\Gamma_s)} \cdot r^\beta F(\theta)  + \sigma,\quad \text{for}\quad \beta\in(0,1),\quad \theta\in (\theta_1,\theta_2),
\end{align}
where $\theta_1:=\frac32 \pi +\arctan{\psi'}(\eta_0)$ and $\theta_2 := 2\pi +\arctan{\psi'}(\eta_0)$, and $F(\theta)$ will be determined later.
By \eqref{eq23}, direct calculations yield that
\begin{align}\label{bar}
  \mathcal{N}\mathcal{F}_1 = &\|f_s\|_{L^\infty(\Gamma_s)}\cdot \Big\{r^{\beta -2} \Big(a_1 \beta (\beta-1)F + 2a_2 (\beta -1)F' + a_3 (F'' + \beta F)\Big)\notag\\
   &\qquad\qquad \qquad  + r^{\beta-1}\Big(a_4 F' + a_5 \beta F\Big)\Big\},
\end{align}
where
\begin{align}
  a_1 = a_{kj} \eta_k \eta_j,\,\, a_2 = a_{kj} \xi_k \eta_j,\,\,
  a_3 = a_{kj} \xi_k \xi_j, \,\,  a_4 = \partial_j(a_{kj}) \xi_k,\,\, a_5 = \partial_j(a_{kj}) \eta_k,
\end{align}
with $k,j=1,2$, $\eta_k = \frac{x_k}{r}$, $\xi_1 = -\sin\theta$, $\xi_2 = \cos\theta$.
Next, on the boundary $\Gamma_N (r_0)$, one has
\begin{align}\label{BD}
 \vec{\ell}_s \cdot D\mathcal{F}_1 =&\displaystyle\frac{1}{\sqrt{1+ (\psi')^2}} \|f_s\|_{L^\infty(\Gamma_s)}\cdot  r^{\beta -1}\notag\\
 &\cdot \Big\{\big(\ell_s^{(1)} \cos\theta + \ell_s^{(2)} \sin\theta\big)\beta F + \big( \ell_s^{(2)} \cos\theta -\ell_s^{(1)} \sin\theta  \big) F'  \Big\},
\end{align}
where $(\ell_s^{(1)}, \ell_s^{(2)})$ defined in \eqref{ell1}-\eqref{ell2}.
Let
\begin{align}
  F(\theta) = 1 - G \exp({H\theta}),
\end{align}
where $G = \exp(-H\theta_2 -1)$, $H = 1+\displaystyle\frac{4\|a_2\|_{L^\infty}}{\lambda}$, and $\lambda$ is defined in \eqref{eq25}.
Obviously, $0<F<1$, $F' = -GH \exp(H\theta)<0$, and $F'' = -G H^2 \exp(H\theta)<0$.
In addition, it is easy to see that $a_1 \beta (\beta-1)F <0$,
\begin{align}\label{2a2eq}
  2a_2 (\beta -1)F'\leq 2\|a_2\|_{L^\infty} \cdot |F'| = 2\|a_2\|_{L^\infty}\cdot  GH \exp(H\theta),
\end{align}
and
\begin{align}\label{a3eq}
  &a_3(F''+\beta F)\leq  \lambda (F''+\beta F)\notag\\
  =& \lambda \Big(-GH^2 \exp(H\theta) + \beta (1 - G \exp(H\theta))   \Big)\notag\\
  \leq& \lambda \Big(-GH^2 \exp(H\theta) + \beta  \Big)\notag\\
  =& - GH\exp(H\theta)\lambda\Big(1+  \frac{4\|a_2\|_{L^\infty}}{\lambda}\Big) +\beta \lambda\notag\\
  =& -4\|a_2\|_{L^\infty} GH\exp(H\theta) + \lambda \Big(\beta - GH\exp(H\theta)\Big).
\end{align}
Let $\beta_0 =GH \exp(H\theta_1 -1)$. For $0< \beta\leq \beta_{0}$ and  sufficiently small $r_0>0$, applying \eqref{2a2eq} and \eqref{a3eq}, then \eqref{bar} yields that $\mathcal{N} \mathcal{F}_1 < 0$. Moreover, by applying \eqref{upperest}, one can deduce that
$\partial_j(a_{kj})$ are bounded in $\Omega_+$. Thus, $|\vec{\ell}_s|$ is bounded. One may assume that $|\vec{\ell}_s|\leq {C}_s$ for some positive constant ${C}_s$. Therefore,
\begin{align}
 \displaystyle\frac{(\ell_s^{(1)} \cos\theta + \ell_s^{(2)} \sin\theta)}{\sqrt{1+ (\psi')^2}}   \leq {C}_s.
\end{align}
In addition,
\begin{align}
 \frac{(\ell_s^{(2)} \cos\theta -\ell_s^{(1)} \sin\theta )}{\sqrt{1+ (\psi')^2}} =\frac{(\ell_s^{(1)}, \ell_s^{(2)})}{\sqrt{1+ (\psi')^2}} \cdot (- \sin\theta ,  \cos\theta)= \vec{\ell}_s \cdot \vec{n} +  \vec{\ell}_s (\vec{\nu}  - \vec{n}),
\end{align}
where $\vec{\nu} =(-\sin\theta, \cos\theta)$, and $\vec{n}$ is the inner normal vector of the boundary $\Gamma_N (r_0)$.
When $r_0>0$ is sufficiently small, one has
\begin{align}
  |\vec{\nu}  - \vec{n} |\leq \delta,
\end{align}
where $\delta > 0$ is small and will be determined below.
Applying \eqref{As}, it holds that
\begin{align}
     \vec{\ell}_s \cdot \vec{n}+  \vec{\ell}_s(\vec{\nu} - \vec{n})\geq \varepsilon_1- | \vec{\ell}_s||\vec{\nu} - \vec{n}|\geq \varepsilon_1- {C}_s\delta.
\end{align}
Let $\delta = \frac{\varepsilon_1}{4{C}_s}$, then one has $ \vec{\ell}_s \cdot \vec{n}+  \vec{\ell}_s(\vec{\nu} - \vec{n})\geq \frac34 \varepsilon_1$.
Therefore, \eqref{BD} implies that
\begin{align}
  \vec{\ell}_s \cdot D\mathcal{F}_1 \leq   \|f_s\|_{L^\infty(\Gamma_s)}\cdot r^{\beta -1}\Big({C}_s \beta - \frac34 \varepsilon_1 GH\Big).
\end{align}
Let $\beta_1 = \displaystyle\frac{\varepsilon_1 GH}{2 {C}_s}$. Denote $\beta_* = \min\{1, \beta_0, \beta_1\}$. For any $0< \beta < \beta_*$, when $r_0>0$ is sufficiently small, one has
\begin{align}
  &\mathcal{N} \mathcal{F}_1 < 0,&\quad &\text{in} \quad \Omega_1(r_0)\\
  &\vec{\ell}_s \cdot  D\mathcal{F}_1 \leq -\|f_s\|_{L^\infty(\Gamma_s)},&\quad &\text{on}\quad \Gamma_N(r_0)\\
  &\mathcal{F}_1 \geq \sigma.&\quad &\text{on} \quad \Gamma_D (r_0)
\end{align}
Therefore, employing the comparison principle and \eqref{zz0}, \eqref{deffs} as well as \eqref{uppsi}, it follows that
\begin{align}\label{eq111}
  \sup\limits_{\overline{\Omega_1(r_0)}}|\theta|\leq \mathcal{F}_1 \leq C(\|f_s\|_{L^\infty(\Gamma_s)} + \sigma)\leq C_1^\sharp \cdot \sigma,
\end{align}
where the constant $ C_1^\sharp$ depends on $C_-, \overline{U}_-$
, $C_0$ and $C_{0s}$.
Then, similarly as done for $\mathcal{F}_1$, one can construct $\mathcal{F}_2$ in the neighborhood of point $P_1$, and obtain that
\begin{align}\label{eq111}
  \sup\limits_{\overline{\Omega_2(r_0)}}|\theta|\leq \mathcal{F}_2 \leq C(\|f_s\|_{L^\infty(\Gamma_s)} + \sigma)\leq C_2^\sharp \cdot \sigma,
\end{align}
where the constant $ C_2^\sharp$ depends on $C_-, \overline{U}_-$
, $C_0$ and $C_{0s}$.
Moreover, define
\begin{align*}
  \Omega_Q(\iota): = \{(\xi, \eta)\in {\mathbb{R}}^2:\psi(\eta)<\xi < \psi(\eta) + \iota, 0<\eta<\eta_0\}\backslash\{\Omega_1(r_0)\cup\Omega_2(r_0) \},
  \end{align*}
  where $\iota>0$ is a small constant.
Now we consider \eqref{eq23} in the domain $\Omega_Q(\iota)$. First, the equation \eqref{eq23} can be rewritten as the following form:
\begin{equation}
 \mathcal{N} \theta := \sum_{k,j = 1}^{2} a_{kj}(U) \theta_{kj}
  +  b_k(U) \theta_k = 0,
\end{equation}
where $b_k\defs \sum\limits_{j=1}^2 \partial_j (a_{kj}(U))$.
Let
\begin{align}
  \omega = M_1 - M_2\exp(\mu (\xi-\psi(\eta))),
\end{align}
where $M_1, M_2, \mu$ are positive constants.
It is easy to check that
\begin{align}
  &\sum_{k,j = 1}^{2}a_{kj}(U) \omega_{kj} + b_k(U) \omega_k\notag\\
  =&- M_2\mu\exp(\mu (\xi-\psi(\eta)))\notag\\
   &\cdot \Big(\mu \big(a_{11} - 2a_{12} \psi'(\eta) + a_{22} (\psi'(\eta))^2 \big)-a_{22} \psi''(\eta) + b_1 - b_2\psi'(\eta)   \Big)\notag\\
   \leq & - M_2\mu\exp(\mu (\xi-\psi(\eta)))\Big( \mu \lambda (1+(\psi')^2) - \Lambda|\psi''(\eta)| - |b_1| - |b_2||\psi'(\eta)|  \Big)\notag\\
   \leq & -C_1<0,
\end{align}
if $\mu$ large and $M_2$ small depending on $\frac{\Lambda}{\lambda}$ and $C_0$. Thus, choosing $M_1$ large such that
\begin{align}
  \omega >0.
\end{align}
In addition,
\begin{align}
   \vec{\ell}_s \cdot D \omega =& - M_2\mu \exp(\mu(\xi-\psi(\eta))) \cdot \vec{\ell}_s \cdot(1, -\psi'(\eta)) \notag\\
  =& - \sqrt{1+(\psi'(\eta))^2} M_2\mu \exp(\gamma (\xi-\psi(\eta))) \cdot\vec{\ell}_s \cdot n_s  \notag\\
  \leq& -C_2 <0.
\end{align}
Let $\theta = \upsilon\cdot\omega$. Then we consider
\begin{align}
   \mathcal{N} \upsilon\defs a_{ij} \upsilon_{ij} + \frac{a_{ij}\omega_j + b_i \omega}{\omega} \upsilon_i + \frac{a_{ij} \omega_i}{\omega}\upsilon_j +
  \frac{a_{ij} \omega_{ij} + b_i \omega_i }{\omega} \upsilon = 0,\quad \text{in} \quad \Omega_Q(\iota).
\end{align}
On the boundary $\Gamma_s\backslash \{P_1, P_4\}$, one has
\begin{align}
  \upsilon = \frac{f_s}{\vec{\ell}_s \cdot D\omega} - \frac{\vec{\ell}_s \cdot D \upsilon}{ \vec{\ell}_s \cdot D\omega} \omega.
\end{align}
Thus, by maximum principle, $\upsilon$ cannot attain its maximum at the interior of $\Omega_Q(\iota)$. Then, if $\upsilon$ attains its maximum at $Q\in \Gamma_s\backslash \{P_1, P_4\}$, one can obtain
\begin{align}
  \upsilon(Q) = \frac{f_s}{\vec{\ell}_s \cdot D\omega}(Q)- \frac{\vec{\ell}_s \cdot D \upsilon}{\vec{\ell}_s \cdot D\omega} \omega(Q) \leq \frac{f_s}{\vec{\ell}_s \cdot D\omega}(Q)\leq \frac{|f_s|}{C_2}.
\end{align}
Then one has
\begin{align}
 \sup\limits_{\overline{\Omega_Q(\iota)}} |\upsilon|\leq \frac{|f_s|}{C_2}.
\end{align}
Thus
\begin{align}\label{OmegaQiota}
 \sup\limits_{\overline{\Omega_Q(\iota)}} |\theta|\leq M_1\frac{|f_s|}{C_2}.
\end{align}
Therefore, applying \eqref{eq111}, \eqref{OmegaQiota} and \eqref{14l}-\eqref{15u}, it follows that
\begin{align}\label{eq111eq}
  \sup\limits_{\overline{\Omega}_+}|\theta| \leq C\cdot(\|f_s\|_{L^\infty(\Gamma_s)} + \sigma)\leq C_0^\sharp \cdot \sigma,
\end{align}
where the constant $ C_0^\sharp$ depends on $C_-, \overline{U}_-$, $C_0$ and $C_{0s}$.

\emph{Step 4}: In this step, we will establish the \emph{a priori} estimates for $p$.
Applying \eqref{g4} and \eqref{eq111eq}, it follows that
\begin{align}\label{psi}
\sup\limits_{\Gamma_s}|\psi'| \leq C_2^\sharp \cdot \sigma,
\end{align}
where the constant $ C_2^\sharp$ depends on $C_1^\sharp, C_-, \overline{U}_-$, $C_0$ and $C_{0s}$.
Then, we rewrite the R-H conditions \eqref{g0}-\eqref{g4} as
\begin{align}
   U = H_3(U_- , \psi'), \quad \text{on} \quad \Gamma_s.
\end{align}
Furthermore, it follows from \eqref{zz0} and \eqref{psi} that
  \begin{align}
    U &= H_3(U_- , \psi') = H_3(\overline{U}_-, 0) + \nabla_{U_-} H_3(\overline{U}_- , 0)(U_- - \overline{U}_-) + \nabla_{\psi'} H_3(\overline{U}_- , 0)\psi'\notag\\
     & = \overline{U}_+ + O(1)\sigma,
  \end{align}
  where the bounded function $O(1)$ depends on $C_-, \overline{U}_-$ and $C_2^\sharp$. Thus, one has
  \begin{align}\label{esgammas}
    U - \overline{U}_+ = O(1)\sigma,\quad \text{on}\quad \Gamma_s.
  \end{align}
Moreover, by \eqref{eq19}, one has
\begin{equation}\label{op3}
D\theta =(1-M^2)\cos\theta \cdot K^{-1} \cdot Dp.
\end{equation}
Taking $(\partial_{\eta}, -\partial_{\xi})$ on both sides of $\eqref{op3}$, one can get the second order elliptic equation of $p$ in the divergence form. By applying the strong maximum principle, one can deduce that $p$ cannot attain its minimum and its maximum values in $\Omega_+$.
In addition, the boundary conditions on $\Gamma_2$ and $\Gamma_4$ can be transformed into
\begin{align}
&\partial_{\vec{n}_2} p - \frac{\sin\theta}{\rho q} \partial_{\vec{\tau}_2}p  =0,\quad \text{on}\quad \Gamma_2\label{pgamma2}\\
&\partial_{\vec{n}_4} p - \frac{\sin\theta}{\rho q} \partial_{\vec{\tau}_4}p =0,\quad \text{on}\quad \Gamma_4\label{pgamma4}
\end{align}
where $\vec{\tau}_2 =(1,0), \vec{n}_2 = (0,1),\vec{\tau}_4 =(-1,0), \vec{n}_4 = (0,-1)$.
Thus, applying the Hopf's lemma, \eqref{pgamma2}-\eqref{pgamma4} yield that $p$ cannot attain its maximum or minimum values on the boundaries $\Gamma_2$ and $\Gamma_4$. Therefore,
\begin{align}\label{ps}
  \sup\limits_{\overline{\Omega}_+}|p-\bar{p}_+|\leq& \sup\limits_{\Gamma_s}|p-\bar{p}_+| + \sup\limits_{\Gamma_3}|p-\bar{p}_+| \notag\\
  =&  \sup\limits_{\Gamma_s}|p-\bar{p}_+| + \sigma  |P_e| \leq C_3^\sharp\cdot \sigma + \sigma |P_e|,
\end{align}
where the constant $C_3^\sharp$ depends on $C_-, \overline{U}_-$ and $C_2^\sharp$.

\emph{Step 5}: On $\Gamma_s$, it follows from the condition \eqref{esgammas} that $s-\bar{s}_+ = O(1)\sigma$ and $q - \bar{q}_+ = O(1)\sigma$. Then by employing \eqref{Conf4}, one has
\begin{align}\label{sseq}
  \sup\limits_{\overline{\Omega}_+}|s-\bar{s}_+| \leq C_4^\sharp \cdot \sigma,
\end{align}
where the constant $C_4^\sharp$ depends on $C_-, \overline{U}_-$ and $C_2^\sharp$.
In addition, from the equation \eqref{Conf3}, one can deduce that
\begin{align}\label{Bereq}
 (\frac12 q^2 + i)(\xi,\eta) =(\frac12 q^2 + i)(\psi(\eta),\eta).
\end{align}
Applying \eqref{ps} and \eqref{sseq}, then \eqref{Bereq} yields that
\begin{align}
  \sup\limits_{\overline{\Omega}_+}|q-\bar{q}_+| \leq C_5^\sharp \cdot \sigma,
\end{align}
where the constant $C_5^\sharp$ depends on $C_-, \overline{U}_\pm$, $C_i^\sharp$, $(i=2,3,4)$, $C_0$ and $C_{0s}$.

\emph{Step 6}: In this step, we raise the regularity of the solution $U$.

By applying \eqref{eq111eq}, \eqref{ps}, and Theorem 8.29 in \cite{GI2011}, one can obtain
\begin{align}\label{w}
 \|p-\bar{p}_+\|_{1,\alpha;\Omega_+}^{(-\alpha;\{P_i\})}+  \|\theta\|_{1,\alpha;\Omega_+}^{(-\alpha;\{P_i\})}
 \leq C_6^\sharp \Big( \sup\limits_{\overline{\Omega}_+}|p-\bar{p}_+|
   + \sup\limits_{\overline{\Omega}_+}|\theta|\Big)\leq C_7^\sharp\sigma.
\end{align}
Applying \eqref{w} and \eqref{zz0}, it follows that
\begin{align}\label{qqs}
   \|q-\bar{q}_+\|_{1,\alpha;\Omega_+}^{(-\alpha;\{\overline{\Gamma_2^+}\cup \overline{\Gamma_4^+}\})}+  \|s-\bar{s}_+\|_{1,\alpha;\Omega_+}^{(-\alpha;\{\overline{\Gamma_2^+}\cup \overline{\Gamma_4^+}\})}\leq C_8^\sharp\sigma.
\end{align}
In addition, by employing \eqref{w}, \eqref{qqs} and \eqref{6}, one has
\begin{align}
   \|\psi' \|_{\mcc^{0,\alpha}(\Gamma_s)}\leq C_9^\sharp\sigma.
\end{align}
Thus, the proof of Lemma \ref{elliptic} is completed.
\end{proof}

\section{Uniqueness of the transonic shock solutions}
Based on the \emph{a priori} estimates in Section 2,
in this section, we are going to show that any piecewise smooth shock solution coincides with the one given in Theorem \ref{thm1:existence} via contraction arguments. It is worth pointing out that one of the key steps in the arguments is to prove the uniqueness of intersection points between the shock front and the nozzle walls.

\subsection{Some properties of supersonic solution $U_-$ ahead of the shock front}
Let $$\psi(\eta_0)\defs \xi_* = \xi_0 + \delta \xi\in (0,L),$$
where $\xi_0\in(0,L)$ is a fixed point and $\delta \xi\in (-\xi_0, L-\xi_0)$ is an unknown point.
Then $\psi(\eta)$ can be rewritten as the following form:
 \begin{align}
   \psi(\eta) = \xi_* - \int_{\eta}^{\eta_0} \psi'(\tau)\dif \tau.
 \end{align}
Denote
\begin{align}
\Omega_{\xi_*}^{-}\defs \{(\xi, \eta)\in \mathbb{R}^2 : 0 < \xi < \xi_*, \,0 < \eta < \eta_0 \}.
\end{align}
Let ${U}_- = \delta {U}_- + \overline{U}_-$ be the supersonic solution in the domain $\Omega_{\xi_*}^{-}$. Then
$\delta {U}_-:=(\delta p_-,\delta \theta_-,  \delta q_-, \delta s_-)$ satisfies the following  equations in $\Omega_{\xi_*}^{-}$:
\begin{align}
  &\partial_\eta \delta p_- + \bar{q}_- \partial_\xi \delta \theta_- = f_1^-,\label{u1-}\\
  &\partial_\eta \delta \theta_- - \frac{1}{\bar{\rho}_-\bar{q}_-}
  \frac{1-\bar{M}_-^2}{\bar{\rho}_-\bar{q}_-^2}\partial_\xi \delta p_- = f_2^-,\label{u2-}\\
  & \partial_\xi (\bar{\rho}_-\bar{q}_- \delta q_- + \delta p_-)=f_3^-,\label{u3-}\\
  & \partial_\xi \delta s_- =0,\label{u4-}
\end{align}
where
\begin{align}
  f_1^- \defs&\frac{\sin\theta_-}{\rho_- q_-}\partial_\xi p_- - (q_- \cos\theta_- \partial_\xi \theta_- - \bar{q}_- \partial_\xi \delta \theta_- ),\\
  f_2^-\defs&\frac{\sin\theta_-}{\rho_- q_-}\partial_\xi \theta_-
  + \Big(\frac{\cos\theta_-}{\rho_- q_-}\frac{1-M_-^2}{\rho_- q_-^2}
  \partial_\xi p_- - \frac{1}{\bar{\rho}_-\bar{q}_-}
  \frac{1-\bar{M}_-^2}{\bar{\rho}_-\bar{q}_-^2} \partial_\xi \delta p_-   \Big),\\
  f_3^-\defs& -(\rho_- q_- \partial_\xi q_- - \bar{\rho}_-\bar{q}_- \partial_\xi \delta q_-).\label{supf3}
\end{align}
By \eqref{eq808}-\eqref{eq909}, it follows that
\begin{align}
	\delta U_-=& \sigma {U}_0(Y_0(\eta)),&\quad& \text{on} \quad  \Gamma_1,\label{12t}\\
	\delta \theta_- =& 0, &\quad &\text{on} \quad \Gamma_2 \cap\overline{\Omega_{\xi_*}^{-}},\\
\delta \theta_- =& \sigma, &\quad &\text{on} \quad \Gamma_4 \cap\overline{\Omega_{\xi_*}^{-}}.\label{zzt}
	\end{align}
Then one has the following lemma.
\begin{lem}\label{lem1}
In $\Omega_{\xi_*}^{-}$, the solution $\delta U_-$ to equations $\eqref{u1-}$-$\eqref{u4-}$, with the initial-boundary conditions \eqref{12t}-\eqref{zzt} satisfies the following relations:
\begin{align}
&\delta s_-(\xi, \eta) =  \sigma s_0 (Y_0(\eta)),\label{deltas-}\\
&(\bar{q}_- \delta q_- +\frac{1}{\bar{\rho}_-}\delta p_-)(\xi,\eta) = \sigma\bar{q}_-  q_0 (Y_0(\eta)) +\sigma \frac{1}{\bar{\rho}_-}p_0 (Y_0(\eta)) + \int_{0}^{\xi}f_3^-(\tau,\eta)\dif \tau.\label{deltaq-}
\end{align}
Moreover, for $\xi_* \in (0, L)$, it holds that
\begin{align}\label{pl0}
  &\frac{\bar{M}_-^2-1}{\bar{\rho}_- \bar{ q}_-^2}\int_{0}^{\eta_0}\delta p_- (\xi_* , \eta)\dif \eta\notag\\
   =& -  \sigma\bar{\rho}_- \bar{q}_-\xi_*
    + \sigma \frac{\bar{M}_-^2-1}{\bar{\rho}_- \bar{ q}_-^2}\int_{0}^{\eta_0} p_0\dif \eta + \bar{\rho}_-\bar{q}_- \int_0^{\eta_0}\int_0^{\xi_*} f_2^- \dif \xi \dif \eta.
\end{align}
\end{lem}

 \begin{proof}
\eqref{deltas-} and \eqref{deltaq-} can be obtained immediately by employing equations \eqref{u3-} and \eqref{u4-}.
Next, to show \eqref{pl0}, one observes that equation \eqref{u2-} implies that
\begin{align}
 \int_{\Omega_{\xi*}^{-}} f_2^- \dif\xi \dif \eta=& \int_{\Omega_{\xi*}^{-}} \Big(\partial_\eta \delta \theta_- - \frac{1}{\bar{\rho}_-\bar{q}_-}
  \frac{1-\bar{M}_-^2}{\bar{\rho}_-\bar{q}_-^2}\partial_\xi \delta p_-\Big) \dif\xi \dif \eta\notag\\
  =&\sigma \xi_*  + \frac{\bar{M}_-^2 -1}{\bar{\rho}_-^2 \bar{q}_-^3}\int_{0}^{\eta_0} \Big(\delta p_- (\xi_*, \eta) - \sigma p_0  \Big) \dif \eta,
\end{align}
 which is exactly \eqref{pl0}.
 \end{proof}

\subsection{Reformulation of the R-H conditions \eqref{tilgi}-\eqref{tilg4} }
Based on the Lemma \ref{elliptic} and Lemma \ref{lem1}, we will reformulate the R-H conditions \eqref{tilgi}-\eqref{tilg4} in this subsection, which will be used to provide information for proving the uniqueness of the positions of the intersection points between the shock front and the upper wall of nozzle.
First, we introduce the following transformation to fix the free boundary $\Gamma_s$ into the straight line $\xi = \xi_0$:
\begin{align*}
\mathcal{T}_{\xi_0} : \begin{cases}
 \tilde{\xi}=L + \displaystyle\frac{L - \xi_0}{L - \psi(\eta)}(\xi - L),\\
 \tilde{\eta} =\eta.
\end{cases}
\end{align*}
Then the domain $\Omega_+$ becomes the following fixed domain:
\begin{align}
{\Omega}_{{\xi}_0} \defs \{(\tilde{\xi}, \tilde{\eta})\in \mathbb{R}^2 : \xi_0 < \tilde{\xi} < L, \,0 < \tilde{\eta} < \eta_0\},
\end{align}
with the boundaries
\begin{align}
&{\Gamma}_s^0 = \{(\tilde{\xi}, \tilde{\eta})\in \mathbb{R}^2 : \tilde{\xi}=\xi_0,\ 0 < \eta < \eta_0\},\\
&{\Gamma}_2^0 = \{(\tilde{\xi}, \tilde{\eta})\in \mathbb{R}^2 :\xi_0 <\tilde{\xi}<L,\ \eta=0 \},\\
&{\Gamma}_3^0 = \{(\tilde{\xi}, \tilde{\eta})\in \mathbb{R}^2 : \tilde{\xi}= L,\ 0 < \eta < \eta_0\},\\
&{\Gamma}_4^0 = \{(\tilde{\xi}, \tilde{\eta})\in \mathbb{R}^2 :\xi_0<\tilde{\xi}< L,\ \eta =\eta_0 \}.
\end{align}
Moreover,
the problem {\bf{$\llbracket\textit{TSPL}\rrbracket$}} can be rewritten as the following problem in the fixed domain ${\Omega}_{\xi_0}$.
 \begin{enumerate}
 \item $\widetilde{\delta {U}}$ satisfies the following equations in ${\Omega}_{\xi_0}$:
 \begin{align}
  &-\frac{1}{\bar{\rho}_+\bar{q}_+}
  \frac{1}{\bar{\rho}_+\bar{q}_+^2}\partial_{\tilde{\xi}} (\bar{M}_+^2\widetilde{\delta p} + {\bar{\rho}_+\bar{q}_+}\widetilde{\delta{q}} )-\partial_{\tilde{\eta}} (\widetilde{\delta\theta})=f_1,\label{f1}\\
  & \partial_{\tilde{\xi}} (\bar{q}_+ \widetilde{\delta\theta})+ \partial_{\tilde{\eta}} \widetilde{\delta p} = f_2,\\
  & \partial_{\tilde{\xi}} \Big( \frac{1}{\bar{\rho}_+}\widetilde{\delta p} +  \bar{q}_+\cdot\widetilde{\delta{q}}\Big)= f_3,\label{f3}\\
  &\partial_{\tilde{\xi}} \widetilde{\delta s} =0,\label{f4}
\end{align}
where
\begin{align}
 & f_1(\widetilde{\delta {U}},\widetilde{\psi'}, \delta \xi)\notag\\
  \defs & \partial_{\tilde{\xi}} \left( -\frac{1}{\bar{\rho}_+\bar{q}_+}
  \frac{1}{\bar{\rho}_+\bar{q}_+^2} \Big(\bar{M}_+^2\widetilde{\delta p} + {\bar{\rho}_+\bar{q}_+}\widetilde{\delta{q}}\Big)-
  \Big(\displaystyle\frac{1}{\tilde{\rho}\tilde{ q}\cos\tilde{\theta}}- \frac{1}{\bar{\rho}_+\bar{q}_+}\Big)  \right)\notag\\
  & -\partial_{\tilde{\eta}} \Big(\widetilde{\delta \theta}-\tan\tilde{\theta} \Big)\notag\\
  &+ \displaystyle\frac{-\delta \xi + \int_{\eta}^{\eta_0}\widetilde{\psi'}(\tau)\dif \tau}{L -\widetilde{\psi}(\eta)}\partial_{\tilde{\xi}} \Big( \frac{1}{\tilde{\rho}\tilde{q} \cos\tilde{\theta}} -\frac{1}{\bar{\rho}_+\bar{q}_+} \Big)+ \frac{(\xi -L)\widetilde{\psi'}}{L-\psi}\partial_{\tilde{\xi} } (\tan\tilde{\theta}),\\
  &f_2(\widetilde{\delta {U}},\widetilde{\psi'}, \delta \xi)\notag\\
  \defs & \partial_{\tilde{\xi}} \Big(\bar{q}_+\widetilde{\delta\theta} - \tilde{q}\sin\tilde{\theta } \Big) + \displaystyle\frac{-\delta \xi + \int_{\eta}^{\eta_0} \widetilde{\psi'}(\tau)\dif \tau}{L -\widetilde{\psi}(\eta)}\partial_{\tilde{\xi}} ( \tilde{q} \sin\tilde{\theta})+ \frac{(\xi -L)\tilde{\psi'}}{L-\widetilde{\psi}}\partial_{\tilde{\xi}}\tilde{ p},\\
  &f_3(\widetilde{\delta {U}})\defs  \partial_{\tilde{\xi}} \left(\frac{1}{\bar{\rho}_+} \widetilde{\delta p} +  \bar{q}_+\cdot \widetilde{\delta{q}}- \Big(\frac12 \tilde{q}^2 + i - (\frac12 \bar{q}_+^2 + \bar{i}_+) \Big) \right).
\end{align}
\item On the nozzle walls,
\begin{align}\label{eq198}
 \widetilde{\delta \theta} =& 0,\quad \text{on}\quad \Gamma_{2}^0 \cap\overline{\Omega_{\xi_0}},\\
 \widetilde{\delta \theta} =& \sigma,\quad \text{on}\quad \Gamma_{4}^0\cap\overline{\Omega_{\xi_0}}.
\end{align}
\item On the exit,
\begin{align}\label{eq199}
  &\widetilde{\delta p} =\sigma P_e,&\quad &\text{on}\quad\Gamma_3^0.
\end{align}

\item On the free boundary ${\Gamma}_s^0$, the R-H conditions \eqref{g1}-\eqref{g4} become
    \begin{align}
  &G_i(\tilde{U}, U_-(\psi(\tilde{\eta}),\tilde{\eta}))=0, \quad i=1,2,3\quad &\text{on}\quad {\Gamma}_s^0 \label{tilgi}\\
  &G_4(\tilde{U}, U_-(\psi(\tilde{\eta}),\tilde{\eta});\psi'(\tilde{\eta}))=0. \quad &\text{on}\quad {\Gamma}_s^0 \label{tilg4}
\end{align}

\end{enumerate}
To simplify the notations, we drop `` $ \tilde{} $ '' in the arguments below.


Let ${U}= \delta {U} + \overline{U}_+$ in $\Omega_{\xi_0}$ with
\begin{align*}
  \delta {U}:=(\delta p,\delta \theta,  \delta q, \delta s).
\end{align*}
Then one has the following lemma.
\begin{lem}\label{3.1s}
  On the shock front $\Gamma_s^0$, one has
  \begin{align}\label{gammasp}
   \displaystyle\frac{\bar{M}_+^2 -1}{\bar{\rho}_+ \bar{q}_+^2} \delta {p}(\xi_0, \eta)  =& \delta f_s\\
\defs &\displaystyle\frac{\bar{M}_-^2 - 1}{ \bar{\rho}_-\bar{q}_-^2} (1 - \kappa)\delta{p}_-(\xi_0 + \delta \xi, \eta) + \kappa_1 \sigma\big(p_0 +\bar{\rho}_- \bar{q}_-  q_0\big) + \kappa_2\sigma  s_0 \notag\\
   & + \kappa_1 \int_0^{\xi_0 + \delta \xi} f_3^- (\tau,\eta)\dif \tau + O(1) \sigma^2\notag,
  \end{align}
where the bounded function $O(1)$ depends on $C_{\pm}, \overline{U}_\pm, P_e$ and $L$.
In addition, see \eqref{kappaeq1}-\eqref{kappa2eq1} for the definitions of the constants $\kappa$, $\kappa_1$ and $\kappa_2$.

\end{lem}

\begin{proof}
Applying \eqref{tilgi}, direct calculations yield that
\begin{align*}
  &\frac{1}{\bar{\rho}_+^2\bar{q}_+ \bar{c}_+^2} \delta p + \frac{1}{\bar{\rho}_+ \bar{q}_+^2}\delta q - \frac{1}{\gamma c_v}\frac{1}{\bar{\rho}_+\bar{q}_+}\delta s\\
  & - \Big( \frac{1}{\bar{\rho}_-^2\bar{q}_- \bar{c}_-^2} \delta p_- + \frac{1}{\bar{\rho}_- \bar{q}_-^2}\delta q_- - \frac{1}{\gamma c_v}\frac{1}{\bar{\rho}_-\bar{q}_-}\delta s_- \Big)=g_1,\\
  &\frac{1}{\bar{\rho}_+ \bar{q}_+}\delta p + \delta q - \bar{p}_+ \Big( \frac{1}{\bar{\rho}_+^2\bar{q}_+ \bar{c}_+^2} \delta p + \frac{1}{\bar{\rho}_+ \bar{q}_+^2}\delta q - \frac{1}{\gamma c_v}\frac{1}{\bar{\rho}_+\bar{q}_+}\delta s\Big)\\
  & - \Big( \frac{1}{\bar{\rho}_- \bar{q}_-}\delta p_- + \delta q_-   - \bar{p}_- \Big( \frac{1}{\bar{\rho}_-^2\bar{q}_- \bar{c}_-^2} \delta p_- + \frac{1}{\bar{\rho}_- \bar{q}_-^2}\delta q_- - \frac{1}{\gamma c_v}\frac{1}{\bar{\rho}_- \bar{q}_-}\delta s_-\Big)\Big)=g_2,\\
  & \frac{1}{\bar{\rho}_+} \delta p + \bar{q}_+ \delta q + \frac{1}{(\gamma-1)c_v}\frac{\bar{p}_+}{\bar{\rho}_+}\delta s - \Big(  \frac{1}{\bar{\rho}_-} \delta p_- + \bar{q}_- \delta q_- + \frac{1}{(\gamma-1)c_v}\frac{\bar{p}_-}{\bar{\rho}_-}\delta s_- \Big)=g_3,
\end{align*}
where
\begin{align}
  g_1=& \frac{1}{[p]}[\frac{v}{u}][v] + \Big( \frac{1}{\rho q \cos\theta} - \frac{1}{\bar{\rho}_+ \bar{q}_+} + \frac{1}{\bar{\rho}_+^2\bar{q}_+ \bar{c}_+^2} \delta p + \frac{1}{\bar{\rho}_+ \bar{q}_+^2}\delta q - \frac{1}{\gamma c_v}\frac{1}{\bar{\rho}_+\bar{q}_+}\delta s  \Big)\notag\\
  &- \Big( \frac{1}{\rho_- q_- \cos\theta_-} - \frac{1}{\bar{\rho}_- \bar{q}_-} + \frac{1}{\bar{\rho}_-^2\bar{q}_- \bar{c}_-^2} \delta p_- + \frac{1}{\bar{\rho}_- \bar{q}_-^2}\delta q_- - \frac{1}{\gamma c_v}\frac{1}{\bar{\rho}_-\bar{q}_-}\delta s_-  \Big),\label{rhg1}\\
  g_2=& -\frac{1}{[p]}[\frac{pv}{u}][v] -\Big(q\cos\theta + \frac{p}{\rho q \cos\theta} - \bar{q}_+ -\frac{\bar{p}_+}{\bar{\rho}_+ \bar{q}_+} - \delta q - \frac{1}{\bar{\rho}_+ \bar{q}_+}\delta p  \Big)\notag\\
   &- \bar{p}_+ \Big( \frac{1}{\bar{\rho}_+^2\bar{q}_+ \bar{c}_+^2} \delta p + \frac{1}{\bar{\rho}_+ \bar{q}_+^2}\delta q - \frac{1}{\gamma c_v}\frac{1}{\bar{\rho}_+\bar{q}_+}\delta s  \Big)\notag\\
   &+ \Big(q_-\cos\theta_- + \frac{p_-}{\rho_- q_- \cos\theta_-} - \bar{q}_- -\frac{\bar{p}_-}{\bar{\rho}_- \bar{q}_-} - \delta q_- - \frac{1}{\bar{\rho}_- \bar{q}_-}\delta p_-  \Big)\notag\\
   &+\bar{p}_- \Big( \frac{1}{\bar{\rho}_-^2\bar{q}_- \bar{c}_-^2} \delta p_- + \frac{1}{\bar{\rho}_- \bar{q}_-^2}\delta q_- - \frac{1}{\gamma c_v}\frac{1}{\bar{\rho}_- \bar{q}_-}\delta s_-  \Big),\\
   g_3=& - \Big(\frac12 q^2 + i - \frac12 \bar{q}_+^2 - \bar{i}_+ - \frac{1}{\bar{\rho}_+}\delta p - \bar{q}_+ \delta q - \frac{1}{(\gamma-1)c_v}\frac{\bar{p}_+}{\bar{\rho}_+}\delta s      \Big)\notag\\
   & + \Big(\frac12 q_-^2 + i_- - \frac12 \bar{q}_-^2 - \bar{i}_- - \frac{1}{\bar{\rho}_-}\delta p_- - \bar{q}_- \delta q_- - \frac{1}{(\gamma-1)c_v}\frac{\bar{p}_-}{\bar{\rho}_-}\delta s_-     \Big).\label{rhg3}
\end{align}
Applying Lemma \ref{elliptic} and \eqref{zz0}, it is easy to see that
\begin{align}
  g_i = O(1) \sigma^2,
\end{align}
where the bounded function $O(1)$ depends on $C_{\pm}, \overline{U}_\pm, P_e$ and $L$.
In addition, notice that
\begin{align}
  U_- (\psi(\eta), \eta) =& U_- (\xi_0 + \delta\xi - \int_{\eta}^{\eta_0} \psi'(\tau)\dif \tau)\notag\\
  =&  U_- (\xi_0 + \delta\xi)+ O(1) \sigma^2,
\end{align}
where $O(1)$ depends on $C_{\pm}, \overline{U}_\pm, P_e$ and $L$.
Employing \eqref{deltas-}-\eqref{deltaq-}, direct calculations yield that
\begin{align}
   &\frac{1}{\bar{\rho}_+ \bar{c}_+^2} \delta p + \frac{1}{\bar{q}_+}\delta q - \frac{1}{\gamma c_v}\delta s\notag\\
  =& \bar{\rho}_+ \bar{q}_+ g_1  + \frac{1}{\bar{\rho}_-\bar{q}_-^2}(\bar{M}_-^2 -1)\delta p_- \notag\\
  & + \frac{1}{\bar{\rho}_-\bar{q}_-^2}(\sigma p_0 + \bar{\rho}_- \bar{q}_- \sigma q_0 + \int_{0}^{\xi_0 + \delta \xi} f_3^- (\tau,\eta)\dif \tau) -\frac{1}{\gamma c_v }\sigma s_0,\label{RHg1}\\
   &\delta p + \bar{\rho}_+ \bar{q}_+\delta q \notag\\
   =& \bar{\rho}_+ \bar{q}_+ g_2 + \bar{\rho}_+ \bar{q}_+ \bar{p}_+ g_1 + \frac{[\bar{p}]}{\bar{\rho}_-\bar{q}_-^2} (\bar{M}_-^2-1)\delta p_- -\frac{[\bar{p}]}{\gamma c_v} \sigma s_0\notag\\
   & - \Big( 1+ \frac{[\bar{p}]}{\bar{\rho}_-\bar{q}_-^2}  \Big) \Big( \sigma p_0 + \bar{\rho}_- \bar{q}_- \sigma q_0 +  \int_{0}^{\xi_0 + \delta\xi} f_3^- (\tau,\eta)\dif \tau \Big),\label{RHg2}\\
    &\delta p + \bar{\rho}_+ \bar{q}_+ \delta q + \frac{1}{(\gamma-1)c_v}\bar{p}_+ \delta s\notag\\
    =& \bar{\rho}_+ g_3 + \frac{\bar{\rho}_+}{\bar{\rho_-}}\Big(   \sigma p_0 + \bar{\rho}_- \bar{q}_- \sigma q_0 +  \int_{0}^{\xi_0 + \delta\xi} f_3^- (\tau,\eta)\dif \tau  \Big).\label{RHg3}
\end{align}
Further calculations yield that \eqref{gammasp}.

\end{proof}

\subsection{Uniqueness of $\psi(\eta_0)$}
With the help of the Lemma \ref{elliptic}, Lemma \ref{lem1} and Lemma \ref{3.1s}, we will prove the uniqueness of the position $\psi(\eta_0)$ of the intersection point between the shock front and the upper wall of the nozzle. First, by applying \eqref{f3}, one has
\begin{align}\label{subst}
  \partial_{\xi}(\bar{\rho}_+\bar{q}_+\delta{q}) = -\partial_{\xi}\delta p +\bar{\rho}_+ f_3.
\end{align}
Substituting \eqref{subst} into \eqref{f1}, it follows that 
\begin{align}\label{recombi}
\partial_{\xi}\Big(\frac{1}{\bar{\rho}_+\bar{q}_+}\frac{1-\bar{M}_+^2} {\bar{\rho}_+\bar{q}_+^2}\delta p \Big)-\partial_\eta (\delta \theta)= f_1 + \frac{1}{\bar{\rho}_+\bar{q}_+^3} f_3.
\end{align}
Applying \eqref{recombi} and \eqref{eq198}-\eqref{tilg4}, it yields that
 \begin{align}\label{solva0}
\mathcal{F}(\delta\xi)
\defs&\frac{1- \bar{M}_+^2}{\bar{\rho}_+^2 \bar{q}_+^3} \int_{0}^{\eta_0} \Big(\delta p(\xi_0, \eta) -  \sigma P_e\Big) \dif \eta + \sigma\cdot (L-\xi_0)\notag\\
  & + \int_{0}^{\eta_0}\int_{\xi_0}^{L} \Big(f_1 + \frac{1}{\bar{\rho}_+\bar{q}_+^3} f_3  \Big)\dif \xi \dif \eta\notag\\
  =& 0.
\end{align}
Then the uniqueness of  $\psi(\eta_0)$  can be obtained by showing the following lemma.
\begin{lem}\label{3.4}
There exists a unique point $\delta \xi\in (-\xi_0, L-\xi_0)$ such that
 \begin{align}\label{uni}
 \mathcal{F}(\delta \xi) = 0.
 \end{align}

\end{lem}

\begin{proof}
Note that
\begin{align}
  f_3 = O(1) \sigma^2,
\end{align}
where the bounded function $O(1)$ depends on $C_{\pm}, \overline{U}_\pm, P_e$ and $L$.
Applying Lemma \ref{elliptic}, direct calculations yield that
\begin{align}\label{f1=}
  f_1 = - \frac{\delta \xi}{L-\xi_0 -\delta \xi}\partial_\xi
   \Big(\frac{1}{\rho q \cos\theta} - \frac{1}{\bar{\rho}_+ \bar{q}_+}   \Big) + O(1)\sigma^2,
\end{align}
where the bounded function $O(1)$ depends on $C_{\pm}, \overline{U}_\pm, P_e$ and $L$. Thus, \eqref{f1=} implies that
\begin{align}\label{ref1}
   &\int_{0}^{\eta_0}\int_{\xi_0}^{L} f_1\dif \xi \dif \eta\notag\\
   =& - \frac{\delta \xi}{L-\xi_0 -\delta \xi}\int_{0}^{\eta_0}\int_{\xi_0}^{L}\partial_\xi
   \Big(\frac{1}{\rho q \cos\theta} - \frac{1}{\bar{\rho}_+ \bar{q}_+}   \Big)\dif \xi \dif \eta + O(1) \sigma^2 \notag\\
   =& \frac{\delta \xi}{L-\xi_0 -\delta \xi}\int_{0}^{\eta_0}\int_{\xi_0}^{L}\partial_\xi
   \Big( \frac{1}{\bar{\rho}_+^2 \bar{q}_+ \bar{c}_+^2}\delta p + \frac{1}{\bar{\rho}_+ \bar{q}_+^2}\delta q  \Big)\dif \xi \dif \eta + O(1) \sigma^2\cdot \delta \xi  + O(1) \sigma^2\notag\\
   =&\frac{\delta \xi}{L-\xi_0 -\delta \xi}\int_{0}^{\eta_0}\frac{1}{\bar{\rho}_+^2 \bar{q}_+ \bar{c}_+^2}\Big(\delta p(L,\eta) - \delta p(\xi_0,\eta)\Big)\dif \eta \notag\\
   &+\frac{\delta \xi}{L-\xi_0 -\delta \xi}\int_{0}^{\eta_0}  \frac{1}{\bar{\rho}_+ \bar{q}_+^2}\Big(\delta q(L,\eta) - \delta q(\xi_0,\eta)\Big)\dif \eta\notag\\
    &+ O(1) \sigma^2\cdot \delta \xi+ O(1) \sigma^2.
\end{align}
From the equation \eqref{f3}, one can deduce that
\begin{align}\label{f3=eq}
&\frac{1}{\bar{\rho}_+ \bar{q}_+^2}\Big(\delta q(L,\eta) - \delta q(\xi_0,\eta)\Big) \notag\\
=& \frac{1}{\bar{\rho}_+^2 \bar{q}_+^3}\Big(\delta p(\xi_0,\eta) - \delta p(L,\eta)\Big) + \frac{1}{\bar{\rho}_+ \bar{q}_+^3}\int_{\xi_0}^{L} f_3 \dif\xi \dif\eta.
\end{align}
Substituting \eqref{f3=eq} into \eqref{ref1} and applying Lemma \ref{elliptic}, one obtains
\begin{align}
  &\int_{0}^{\eta_0}\int_{\xi_0}^{L} f_1\dif \xi \dif \eta\notag\\
  =&\frac{\delta \xi}{L-\xi_0 -\delta \xi}\int_{0}^{\eta_0}\frac{\bar{M}_+^2 -1}{\bar{\rho}_+^2 \bar{q}_+^3} \Big(\delta p(L,\eta) - \delta p(\xi_0,\eta)\Big)\dif \eta \notag\\
  &+ O(1) \sigma^2\cdot \delta \xi+ O(1) \sigma^2.
\end{align}
Therefore, \eqref{solva0} yields that
\begin{align}\label{mathF}
  \mathcal{F}(\delta \xi) = &\Big(1+\frac{\delta \xi}{L-\xi_0 -\delta \xi}\Big)\frac{1-\bar{M}_+^2}{\bar{\rho}_+^2 \bar{q}_+^3}\int_{0}^{\eta_0} \Big( \delta p(\xi_0,\eta) - \sigma P_e \Big)\dif \eta \notag\\
  & + (L-\xi_0)\sigma +  O(1) \sigma^2\cdot \delta \xi+ O(1) \sigma^2.
\end{align}
By applying \eqref{gammasp} and \eqref{pl0}, one has
\begin{align}\label{substi}
  &\frac{1-\bar{M}_+^2}{\bar{\rho}_+ \bar{q}_+^2}\int_{0}^{\eta_0} \delta p(\xi_0,\eta)\dif \eta \notag\\
  =& \frac{1-\bar{M}_-^2}{\bar{\rho}_- \bar{q}_-^2}(1-k)\int_{0}^{\eta_0}\delta p_-(\xi_0+\delta \xi,\eta)\dif \eta - \sigma \int_{0}^{\eta_0} (\kappa_1p_0 + \kappa_1 \bar{\rho}_-\bar{q}_- q_0 + \kappa_2 s_0) \dif \eta \notag\\
  & - \kappa_1 \int_0^{\eta_0} \int_0^{\xi_0 +\delta \xi} f_3^-\dif \xi \dif \eta\notag\\
  =& (1-k) \Big(\sigma\bar{\rho}_- \bar{q}_-(\xi_0 + \delta \xi)
     -\sigma \frac{\bar{M}_-^2-1}{\bar{\rho}_- \bar{ q}_-^2}\int_{0}^{\eta_0} p_0\dif \eta - \bar{\rho}_-\bar{q}_- \int_0^{\eta_0}\int_0^{\xi_0 +\delta \xi} f_2^- \dif \xi \dif \eta \Big)\notag\\
     &- \sigma \int_{0}^{\eta_0} (\kappa_1p_0 + \kappa_1 \bar{\rho}_-\bar{q}_- q_0 + \kappa_2 s_0) \dif \eta - \kappa_1 \int_0^{\eta_0} \int_0^{\xi_0 +\delta \xi} f_3^-\dif \xi \dif \eta.
\end{align}
Substituting \eqref{substi} into \eqref{mathF}, one has
\begin{align}\label{monooF}
 & \mathcal{F}(\delta \xi) \notag\\
 =& \sigma\frac{L-\xi_0}{L-\xi_0-\delta\xi}(1-k)(\xi_0 +\delta \xi) +  (L-\xi_0)\sigma\notag\\
  &-\sigma \frac{L-\xi_0}{L-\xi_0-\delta\xi} \frac{1}{\bar{\rho}_+ \bar{q}_+}\int_{0}^{\eta_0}\Big( \big( \kappa_1+ (1-k)\frac{\bar{M}_-^2 -1}{\bar{\rho}_- \bar{q}_-^2}\big)p_0 + \kappa_1\bar{\rho}_-\bar{q}_-q_0 + \kappa_2 s_0\Big)\dif \eta\notag\\
  & - \sigma\frac{L-\xi_0}{L-\xi_0-\delta\xi}\frac{1-\bar{M}_+^2}{\bar{\rho}_+^2 \bar{q}_+^3}\int_{0}^{\eta_0}  P_e \dif \eta +  O(1) \sigma^2\cdot \delta \xi+ O(1) \sigma^2\notag\\
  =& \sigma\frac{L-\xi_0}{L-\xi_0-\delta\xi}\widetilde{\mathcal{F}}(\delta \xi),
\end{align}
where
\begin{align}
  \widetilde{\mathcal{F}}(\delta \xi)\defs& (1-k)(\xi_0 +\delta \xi) + L-\xi_0-\delta \xi  \notag\\
  &- \frac{1}{\bar{\rho}_+ \bar{q}_+}\int_{0}^{\eta_0}\Big( \big( \kappa_1+ (1-k)\frac{\bar{M}_-^2 -1}{\bar{\rho}_- \bar{q}_-^2}\big)p_0 + \kappa_1\bar{\rho}_-\bar{q}_-q_0 + \kappa_2 s_0\Big)\dif \eta\notag\\
  & - \frac{1-\bar{M}_+^2}{\bar{\rho}_+^2 \bar{q}_+^3}\int_{0}^{\eta_0}  P_e \dif \eta +  O(1) \sigma\cdot (1+\delta \xi)\frac{L-\xi_0-\delta\xi}{L-\xi_0}.
\end{align}
It is easy to check that
\begin{align}\label{monoF}
   \widetilde{\mathcal{F}}'(\delta \xi) = -k + O(1)\sigma < 0,
\end{align}
for the sufficiently small constant $\sigma > 0 $, which yields that $ \widetilde{\mathcal{F}}(\delta \xi)$ is a strictly decreasing function. Therefore, there exists a unique $\delta \xi$ such that $\widetilde{\mathcal{F}}(\delta \xi)=0$. Furthermore, there exists a unique $\delta \xi\in (-\xi_0, L-\xi_0)$ such that ${\mathcal{F}}(\delta \xi)=0$.

\end{proof}

\subsection{Uniqueness of the transonic shock solutions}
Lemma \ref{3.4} yields that the uniqueness of the shock position $\xi_*$. Therefore, one can fix the shock front $ \{\xi =\psi(\eta)\}$ into $\{\xi = \xi_*\}$. Let
\begin{align*}
\mathcal{T}_{\xi_*} : \begin{cases}
 \tilde{\xi}= L + \displaystyle\frac{L - \xi_*}{L - \psi(\eta)}(\xi - L),\\
 \tilde{\eta} =\eta.
\end{cases}
\end{align*}
Then the domain $\Omega_+$ becomes
\begin{align}
{\Omega}_{\xi_*} \defs \{(\tilde{\xi}, \tilde{\eta})\in \mathbb{R}^2 : \xi_* < \tilde{\xi} < L, \,0 < \tilde{\eta} < \eta_0\},
\end{align}
with the boundaries
\begin{align}
&{\Gamma}_s^* = \{(\tilde{\xi}, \tilde{\eta})\in \mathbb{R}^2 : \tilde{\xi}=\xi_*,\ 0 < \tilde{\eta} < \eta_0\},\\
&{\Gamma}_2^* = \{(\tilde{\xi}, \tilde{\eta})\in \mathbb{R}^2 :\xi_* <\tilde{\xi}<L,\ \tilde{\eta}=0 \},\\
&{\Gamma}_3^* = \{(\tilde{\xi}, \tilde{\eta})\in \mathbb{R}^2 : \tilde{\xi}= L,\ 0 < \tilde{\eta} < \eta_0\},\\
&{\Gamma}_4^* = \{(\tilde{\xi}, \tilde{\eta})\in \mathbb{R}^2 :\xi_* <\tilde{\xi}< L,\ \tilde{\eta} =\eta_0 \}.
\end{align}
In addition, the intersection points $P_i$ (see Figure \ref{fig:3L}) become $\widetilde{P}_1 = (\xi_*,0)$ and $\widetilde{P}_j = P_j$ $(j=2,3,4)$.
Suppose that there exist two pair of solutions $U_1\defs(p_1,\theta_1,q_1,s_1 )^\top$ and $U_2\defs (p_2, \theta_2, q_2, s_2)^\top$ to the problem {\bf {$\llbracket\textit{TSPL}\rrbracket$}} with the same boundary conditions, and the corresponding shock fronts are $\psi_1(\eta)$ and $\psi_2(\eta)$ respectively, where
 \begin{align}\label{plm}
  \psi_1(\eta) = \xi_* - \int_{\eta}^{\eta_0} \psi_1'(\tau)\dif \tau,\quad
  \psi_2(\eta)=\xi_* - \int_{\eta}^{\eta_0} \psi_2'(\tau)\dif \tau.
\end{align}
 Let
 \begin{align*}
 (P, \Theta, Q,S) =(p_2 - p_1, \theta_2 - \theta_1, q_2 - q_1, s_2-s_1).
 \end{align*}
Applying \eqref{Conf1}-\eqref{Conf4} and dropping `` $ \tilde{} $ '', direct calculations yield that
  \begin{align}
 &\partial_{\xi}\Big(\frac{1}{\bar{\rho}_+\bar{q}_+}\frac{1-\bar{M}_+^2} {\bar{\rho}_+\bar{q}_+^2} P \Big)-\partial_\eta \Theta= F_1 + \frac{1}{\bar{\rho}_+\bar{q}_+^3} F_3,\label{F1}\\
  & \partial_{{\xi}} (\bar{q}_+\Theta)+ \partial_{{\eta}} P = F_2,\label{F2}\\
  & \partial_{{\xi}} \Big( \frac{1}{\bar{\rho}_+} P+  \bar{q}_+ Q\Big)= F_3,\label{F3}\\
  &\partial_{{\xi}} S=0,\label{F4}
\end{align}
where
\begin{align}
   F_1\defs & \partial_{{\xi}} \left( -\frac{1}{\bar{\rho}_+\bar{q}_+}
  \frac{1}{\bar{\rho}_+\bar{q}_+^2} \Big(\bar{M}_+^2 p_2 + {\bar{\rho}_+\bar{q}_+} q_2\Big)-
  \Big(\displaystyle\frac{1}{{\rho_2}{q}_2\cos{\theta}_2}- \frac{1}{\bar{\rho}_+\bar{q}_+}\Big)  \right)\notag\\
   & -\partial_{{\xi}} \left( -\frac{1}{\bar{\rho}_+\bar{q}_+}
  \frac{1}{\bar{\rho}_+\bar{q}_+^2} \Big(\bar{M}_+^2 p_2 + {\bar{\rho}_+\bar{q}_+} q_1\Big)-
  \Big(\displaystyle\frac{1}{{\rho_1}{q}_1\cos{\theta}_1}- \frac{1}{\bar{\rho}_+\bar{q}_+}\Big)  \right)\notag\\
  & -\partial_{{\eta}} (\theta_2 -\tan \theta_2)+\partial_{{\eta}} (\theta_1 -\tan \theta_1)\notag\\
  &+ \displaystyle\frac{\int_{\eta}^{\eta_0}{\psi_2'}(\tau)\dif \tau}{L -{\psi_2}(\eta)}\partial_{{\xi}} \Big( \frac{1}{{\rho_2}q_2 \cos{\theta}_2} -\frac{1}{\bar{\rho}_+\bar{q}_+} \Big)+ \frac{(\xi -L){\psi_2'}}{L-\psi}\partial_{{\xi} } (\tan{\theta}_2)\notag\\
  &- \displaystyle\frac{\int_{\eta}^{\eta_0}{\psi_1'}(\tau)\dif \tau}{L -{\psi_1}(\eta)}\partial_{{\xi}} \Big( \frac{1}{{\rho_1}q_1 \cos{\theta}_1} -\frac{1}{\bar{\rho}_+\bar{q}_+} \Big)- \frac{(\xi -L){\psi_1'}}{L-\psi}\partial_{{\xi} } (\tan{\theta}_1),\\
  F_2\defs & \partial_{{\xi}}(\bar{q}_+ \theta_2
   -{q}_2\sin {\theta}_2 ) + \displaystyle\frac{\int_{\eta}^{\eta_0} {\psi_2'}(\tau)\dif \tau}{L -{\psi}_2(\eta)}\partial_{{\xi}} ( {q}_2 \sin {\theta}_2)+ \frac{(\xi -L){\psi_2'}}{L-{\psi_2}}\partial_{{\xi}}{p}_2\notag\\
   &- \partial_{{\xi}}(\bar{q}_+ \theta_1
   -{q}_1\sin {\theta}_1 ) - \displaystyle\frac{\int_{\eta}^{\eta_0} {\psi_1'}(\tau)\dif \tau}{L -{\psi}_1(\eta)}\partial_{{\xi}} ( {q}_1 \sin {\theta}_1)-  \frac{(\xi -L){\psi_1'}}{L-{\psi_1}}\partial_{{\xi}}{p}_1,\\
  F_3\defs & \partial_{{\xi}} \left(\frac{1}{\bar{\rho}_+} p_2 +  \bar{q}_+\cdot q_2- \Big(\frac12 {q_2}^2 + i_2 - (\frac12 \bar{q}_+^2 + \bar{i}_+) \Big) \right)\notag\\
  & -\partial_{{\xi}} \left(\frac{1}{\bar{\rho}_+} p_1 +  \bar{q}_+\cdot q_1 - \Big(\frac12 {q_1}^2 + i_1 - (\frac12 \bar{q}_+^2 + \bar{i}_+) \Big) \right).
\end{align}
Moreover, one has
\begin{align}
  \Theta =& 0,&\quad &\text{on}\quad (\Gamma_{2}^* \cap\overline{\Omega_{\xi_*}})\cup
  (\Gamma_{4}^*\cap\overline{\Omega_{\xi_*}}),\\
 P =& 0 ,&\quad &\text{on}\quad\Gamma_3^*.
\end{align}
In addition, on the free boundary ${\Gamma}_s^*$, one has
    \begin{align}
  &G_i({U}_2,  U_-^{(2)}) -G_i({U}_1,  U_-^{(1)})  =0, \quad i=1,2,3, \quad &\text{on}\quad {\Gamma}_s^* \label{G123F}\\
  &G_4(U_2,  U_-^{(2)};\psi_2'({\eta})) -  G_4(U_1, U_-^{(1)};\psi_1'({\eta}))=0, \quad &\text{on}\quad {\Gamma}_s^* \label{G4F}
\end{align}
where
\begin{align}
  U_-^{(i)} \defs U_-(\xi_* -\int_{\eta}^{\eta_0}\psi_i'(\tau)\dif \tau,{\eta}), \quad i=1,2.
\end{align}

Applying \eqref{zz0} and Lemma \ref{elliptic} as well as the elliptic theory and hyperbolic theory, one can easily obtain
\begin{lem}\label{contraction}
 Under the assumptions of Theorem \ref{thm2}, there exists a sufficiently small positive constant $\sigma_2$, depending on $\overline{U}_{\pm}$, $L$ and $\epsilon$, such that for any $0< \sigma \leq \sigma_2$, if there are two solutions $(U_1,\psi_1)$ and $(U_2,\psi_2)$ to the problem {\bf {$\llbracket\textit{TSPL}\rrbracket$}}, then the following estimates hold:
  \begin{align}
    &\| U_2 - U_1\|_{({\Omega}_{\xi_*};\widetilde{P}_i)}+ \|\psi_2^{'} -\psi_1^{'}\|_{\mcc^{0,\alpha} (\Gamma_s^*)}\notag\\
  \leq& C \sigma \Big(\| U_2 - U_1\|_{({\Omega}_{\xi_*};\widetilde{P}_i)}
  + \|\psi_2^{'} -\psi_1^{'}\|_{\mcc^{0,\alpha} (\Gamma_s^*)}\Big),
  \end{align}
  where the constant $C$ depends on $\overline{U}_\pm$, $L$ and $\alpha$.
  \end{lem}

Hence Theorem \ref{thm2} can be proven by lemma \ref{contraction} easily if $\sigma$ is sufficiently small, \emph{i.e.}, when $\sigma$ is sufficiently small, it follows from Lemma \ref{contraction} that $\| U_2 - U_1\|_{({\Omega}_{\xi_*};\widetilde{P}_i)}+ \|\psi_2^{'} -\psi_1^{'}\|_{\mcc^{0,\alpha} (\Gamma_s^*)}=0$. Then solution $(U; \psi)$ is unique.
%

\section*{Acknowlegements}
The research of Beixiang Fang was supported in part by Natural Science
Foundation of China under Grant Nos. 12331008 and 11971308, and the Fundamental Research Funds for the Central Universities. The research of Xin Gao was supported in part by SJTU Overseas Study Grants for Excellent PhD Students (No.WF610560514).
The research of Wei Xiang was supported in part by the
Research Grants Council of the HKSAR, China (Project No. CityU 11304820, CityU 11300021, CityU 11311722, and CityU 11305523) and in part by the Research Center for Nonlinear Analysis of the Hong Kong Polytechnic University.

%

\end{document}